\documentclass[amstex, 11pt]{amsart}
\usepackage[dvipdfmx]{graphicx}
\usepackage[all]{xy}
\usepackage{here}
\newtheorem{thm}{Theorem}[section]
\newtheorem{Pro}{Proposition}[section]
\newtheorem{lemma}{Lemma}[section]
\newtheorem{Cor}{Corollary}[section]
\newtheorem{Mythm}{Theorem}

\theoremstyle{definition}
\newtheorem{Def}{Definition}[section]

\theoremstyle{definition}
\newtheorem{Rem}{Remark}[section]
\numberwithin{equation}{section}

\begin{document}
\title[Holomorphic motion and monodromy]{Extension of
    holomorphic motions \\ and monodromy}
\author{
    Hiroshige Shiga    
}
\address{Department of Mathematics,
Tokyo Institute of Technology \\
O-okayama, Meguro-ku Tokyo, Japan} 

\email{shiga@math.titech.ac.jp}
\date{\today}    
\thanks{
The author was partially supported by the Ministry of Education, Science, Sports
and Culture, Japan;
Grant-in-Aid for Scientific Research (B), 16H03933, 2016--2020.}

\begin{abstract}
  Let $E$ be a closed set in the Riemann sphere $\widehat{\mathbb{C}}$. 
  We consider a holomorphic motion $\phi$ of $E$ over a complex manifold $M$, that is, a holomorphic family of injections on $E$ parametrized by $M$.
  It is known that if $M$ is the unit disk $\Delta$ in the complex plane, then any holomorphic motion of $E$ over $\Delta$ can be extended to a holomorphic motion of the Riemann sphere over $\Delta$.
  
  In this paper, we consider conditions under which a holomorphic motion of $E$ over a non-simply connected Riemann surface $X$ can be extended to a holomorphic motion of $\widehat{\mathbb{C}}$ over $X$. Our main result shows that a topological condition, the triviality of the monodromy, gives a necessary and sufficient condition for a holomorphic motion of $E$ over $X$ to be extended to a holomorphic motion of $\widehat{\mathbb{C}}$ over $X$.
  
  We give topological and geometric conditions for a holomorphic motion over a Riemann surface to be extended. We also apply our result to a lifting problem for holomorphic maps to Teichm\"uller spaces. 
\end{abstract}
\maketitle
\section{Introduction}
A family of injections depending holomorphically on complex parameters is called a {\em holomorphic motion} (see Definition \ref{dfn:holomorphic motion} for the precise definition), which is a powerful tool for complex dynamics, the deformation theory of Kleinian groups and Teichm\"uller theory. 
It first appeared in the $\lambda$-lemma of a paper \cite{Mane1983} by Man\'e, Sad and Sullivan.
The lemma says that every holomorphic motion of a set $E$ in the Riemann sphere $\widehat{\mathbb C}$ is extended to a holomorphic motion of the closure of $E$.
Since then, the theory of holomorphic motions has been rapidly developed. 

In fact, the $\lambda$-lemma was improved by Bers-Royden(\cite{Bers1986a}) and Sullivan-Thurston (\cite{Sullivan1986}); Slodkowski (\cite{Slodkowski1991}, \cite{Slodkowski1995}) showed that every holomorphic motion of a closed set in $\widehat{\mathbb{C}}$ over the unit disk $\Delta\subset\mathbb{C}$ is extended to a holomorphic motion of $\widehat{\mathbb{C}}$ over $\Delta$.
Moreover, Earle, Kra and Krushkal (\cite{Earle1994}) proved that if a holomorphic motion over $\Delta$ is $G$-equivariant for a subgroup $G$ of M\"obius transformations (see Definition \ref{dfn:G-equivariance}), then it is extended to a holomorphic motion of $\widehat{\mathbb{C}}$ which is also $G$-equivariant.  
 
 On the other hand, if the parameter space of a holomorphic motion is not simply connected, then the holomorphic motion cannot be extended to a holomorphic motion of $\widehat{\mathbb{C}}$ in general (see \cite{BJMS2012} for related results). 

 In this paper, we consider conditions of a holomorphic motion over a non-simply connected Riemann surface $X$ under which the holomorphic motion can be extended to a holomorphic motion of $\widehat{\mathbb{C}}$ over $X$.
 In fact, we give a necessary and sufficient condition for a holomorphic motion over a Riemann surface to be extended to a holomorphic motion of $\widehat{\mathbb C}$ over the Riemann surface.
 
 We make some examples of holomorphic motions including a counter-example to a claim of Chirka's paper \cite{Chirka2004}.
 
 In the final section, we give an application of our results to a lifting problem for holomorphic maps on a Riemann surface to Teichm\"uller space. 
 We also give topological and geometric conditions for a holomorphic motion over a Riemann surface to be extended.
 
 \medskip
 
{\bf Acknowledgment.} The author thanks Professor S. Mitra for his valuable suggestions and comments.

\section{Preliminaries and statements of main results}
 We begin with the precise definition of holomorphic motions.
\begin{Def}
\label{dfn:holomorphic motion}
	Let $M$ be a complex manifold (or complex Banach manifold in general,) with a basepoint $p_0$ and $E$ a subset of $\widehat{\mathbb{C}}$. 
	A map $\phi : M\times E\to \widehat{\mathbb{C}}$ is called a \textit{holomorphic motion} of $E$ over $M$ if it satisfies the following conditions.
	\begin{enumerate}
  \item $\phi(p_0, z)=z$ for every $z\in E$;
  \item for each $p\in M$, $\phi_{p}(\cdot) :=\phi(p, \cdot)$ is an injection on $E$;
  \item for each $z\in E$, $\phi (\cdot, z)$ is holomorphic on $M$;
\end{enumerate}
Furthermore, the holomorphic motion $\phi$ is called \emph{normalized} if $E$ contains $0. 1$ and $\infty$ and $\phi_{p}$ fixes $0, 1$ and $\infty$ for any $p\in M$.
\end{Def}
It is always possible that a given holomorphic motion is changed to a normalized one by conjugating M\"obius transformations.  Throughout this paper, we always assume that a holomorphic motion is normalized.

As we have noted above, we see from the $\lambda$-lemma (\cite{Mane1983}) that every holomorphic motion of a set $E$ is extended to a holomorphic motion of the closure of $E$.
Hence, without loss of generality, we may assume that the set $E$ is closed.

An important problem on holomorphic motions which we consider in this paper, is the extension problem, that is, we consider conditions under which the holomorphic motion of a subset $E$ over $M$  is extended to a holomorphic motion of $\widehat{\mathbb{C}}$ over $M$. 

We have already known some answers to the problem (\cite{Bers1986a}, \cite{Slodkowski1991}, \cite{Slodkowski1995}).
\begin{thm}
\label{thm:old}
	Let $\phi : M\times E\to \widehat{\mathbb{C}}$ be a holomorphic motion of $E$ over a complex manifold $M$.
	\begin{enumerate}
	\item For each $p\in M$, there exists a neighborhood $U_{p}$ of $p$ such that the restricted holomorphic motion $\phi_{U_{p}} : U_{p}\times E\to \widehat{\mathbb{C}}$ can be extended to a holomorphic motion of $\widehat{\mathbb{C}}$ over $U_{p}$;
  \item if $\textrm{dim }M=1$ and simply connected, then $\phi$ is extended to a holomorphic motion of $\widehat{\mathbb{C}}$ over $M$;
  \item the above statement (2) is not true in general if $dim M>1$; 
for $n>1$ and for any finite subset $E$ of $\widehat{\mathbb{C}}$ consisting of $(n+3)$ points, there exist an $n$-dimensional simply connected complex manifold $M$ and a holomorphic motion $\phi : M\times E\to \widehat{\mathbb{C}}$ such that $\phi$ cannot be extend to a holomorphic motion of $\widehat{\mathbb{C}}$ over $M$.
\end{enumerate}
\end{thm}
\begin{Rem}
\begin{enumerate}
  \item The third statement in Theorem \ref{thm:old} is obtained from a result that there are no holomorphic sections from Teichm\"uller space to the space of Beltrami coefficients if the dimension of the Teichm\"uller space is more than one (see \cite{Earle2000}).
  \item From the uniformization theorem, a simply connected one-dimensional complex manifold $M$ is biholomorphically equivalent to either the unit disk, the complex plane $\mathbb{C}$ or the Riemann sphere $\widehat{\mathbb{C}}$. 
	However, if $M$ is either $\mathbb{C}$ or $\widehat{\mathbb{C}}$, then the holomorphic motion $\phi$ is always {\em trivial}, that is, $\phi(p, z)=z$ for any $(p, z)\in M\times E$.
	More precisely, a simply connected complex manifold $M$ does not admit non-constant bounded holomorphic functions if and only if every holomorphic motion of $E$ over $M$ is trivial (cf.\ \cite{Mitra2010}).
\end{enumerate}
	\end{Rem}
In this paper, we deal with the case where the complex manifold $M$ is one-dimensional, that is, a Riemann surface, but not simply connected.  
While we consider the extension problem for a holomorphic motion of $E$ over a Riemann surface,
a key concept is {\it{the triviality of the monodromy}} defined for a holomorphic motion over a general complex manifold. 
Let us explain the triviality.

Let $\phi : M\times E\to \widehat{\mathbb{C}}$ be a holomorphic motion of a set $E \subset \widehat{\mathbb{C}}$ over a complex manifold $M$ with a basepoint $p_{0}$.
For a finite subset $E'$ of $E$, we consider the restriction $\phi |{E'}$ of the holomorphic motion $\phi$ to $M\times E'$. 
Obviously, $\phi |{E'}$ is also a holomorphic motion of $E'$ over $M$.
From Theorem \ref{thm:old} (1), $\phi |{E'}$ gives a holomorphic fibration of Riemann spheres with $n$ punctures over $M$, where $n=\sharp E'$, the cardinal number of the set $E$.
Thus, for any $\sigma\in \pi_{1}(M, p_0)$, there exists $\rho_{E'}(\sigma)\in \textrm{Mod}(0, n)$ as the monodromy image for $\sigma$ of the fibration, where $\textrm{Mod}(0, n)$ is the set of mapping classes of $n$-punctured sphere. 
Then, we say that the monodromy of the holomorphic motion $\phi$ is \emph{trivial} if $\rho_{E'}(\sigma)=[id]$ for any finite subset $E'$ of $E$ and for any $\sigma\in\pi_{1}(M, p_0)$.

The triviality of the monodromy is described from viewpoint of the theory of braids.
Let $\gamma$ be a closed curve which represents $\sigma\in \pi_1 (M, p_0)$. 
For a parametrization $\gamma : [0, 1]\to M$ of the curve $\gamma$, we may have a braid $\{(t, \phi (\gamma (t), z))\}_{t\in [0, 1], z\in E'}$ of $n$ strands from $E'$.
It is known that the braid gives the same $\rho_{E'}(\sigma)\in \textrm{Mod}(0, n)$ as the monodromy image of the braid for $\sigma$  (cf. \cite{Kassel2008}).
Then, we may say that the monodromy of $\phi$ is trivial if and only if the monodromy of the braid is trivial for any finite subset $E'$ of $E$ and for any $\sigma\in \pi_1 (M, p_0)$.

We consider another kind of monodromy for a holomorphic motion $\phi : M\times E\to \widehat{\mathbb{C}}$ which is called the \emph{trace monodromy}. 
Take distinct four points $z_1, z_2, z_3, z_4$ and put $\mathbf{q}:= \{z_1, z_2, z_3, z_4\}$. 
Then, 
\begin{equation}
	\label{eqn:tracemono}
	\phi^{\mathbf{q}}(\cdot, z_j):=\phi (\cdot, z_j), \quad (j=1, 2, 3, 4)
\end{equation} 
defines a holomorphic motion of the set $\mathbf{q}$ over $X$. We say that the trace monodromy of $\phi$ is trivial if the monodromy of $\phi^{\mathbf{q}}$ is trivial for any $\mathbf{q}$ of distinct four points of $E$.

	The above definition of the trace monodromy has essentially the same idea in common with the definition given in \cite{BJMS2012}, while the condition of the definition here is stronger.
	
	It is not hard to see that if a holomorphic motion $\phi$ of $E$ is extended to a holomorphic motion of $\widehat{\mathbb{C}}$, then both the monodromy and the trace monodromy of $\phi$ are trivial. In our previous paper \cite{BJMS2012}, we posed a question that whether the converse is true or not. 
	We give the affirmative answer if $E$ consists of four points (\cite{BJMS2012}, Theorem B). 
	The first theorem of our paper gives the affirmative answer for the monodromy.
\begin{Mythm}
\label{thm:mythm1}
	Let $\phi : X\times E\to\widehat{\mathbb{C}}$ be a holomorphic motion of $E$ over a Riemann surface $X$. 
	Then, $\phi$ can be extended to a holomorphic motion of $\widehat{\mathbb{C}}$ over $X$ if and only if the monodromy of $\phi$ is trivial.
\end{Mythm}
\begin{Rem}
	In \cite{BJMS2012} Theorem D, we have proved the above theorem for a holomorphic motion over a Riemann surface which is a complement of a certain small compact set in the unit disk. Hence, Theorem \ref{thm:mythm1} is a generalization of the theorem.
\end{Rem}
	Chirka \cite{Chirka2004} considers a condition similar to the triviality of the monodromy and the trace monodromy.

Let $\phi : X\times E\to \widehat{\mathbb{C}}$ be a holomorphic motion of $E$ over a Riemann surface $X$. 
Take distinct points $z_1, z_2$ in $E$ and a smooth closed curve $\alpha$ in $X$; we consider the variation of argument of $\phi (\cdot, z_1)-\phi (\cdot, z_2)$ along $\alpha$, that is,
\begin{equation*}
	n(\alpha ; z_1, z_2)=\int_{\alpha}d \textrm{arg}(\phi (\cdot, z_1)-\phi (\cdot, z_2)).
\end{equation*}
Since $\alpha$ is a closed curve, $n(\alpha ; z_1, z_2)$ is in $2\pi \mathbb{Z}$. 

 In \cite{Chirka2004}, he claimed that
if $n(\alpha ; z, z')=0$ for any distinct points $z, z'$ of $E$ and for any smooth closed curve $\alpha$ in $X$,
then $\phi$ can be extended to a holomorphic motion of $\widehat{\mathbb{C}}$ over $X$.
Unfortunately, the claim is not true; we may construct a counter example to this claim.
We also give another example including a negative answer to the above question for the trace monodromy.

\begin{Mythm}
	\label{thm:mythmTrace}
	Let $E=\{0, 1, \infty, z_0, \dots , z_n\}$ be a finite subset of $\widehat{\mathbb{C}}$ $(n\geq 0)$.
	\begin{enumerate}
  \item There exist a Riemann surface $X$ and a holomorphic motion $\phi : X\times E\to \widehat{\mathbb{C}}$ such that $n(\alpha; z, z')=0$ for any closed curve $\alpha$ in $X$ and for any distinct points $z, z'$ in E but the holomorphic motion $\phi$ cannot be extended to a holomorphic motion of $\widehat{\mathbb{C}}$ over $X$;
  \item there exist a Riemann surface $X$ and a holomorphic motion $\phi : X\times E\to \widehat{\mathbb{C}}$  such that for any proper subset $E'$ of $E$, the restriction $\phi |{E'}$ of $\phi$ to $X\times E'$ can be extended to a holomorphic motion of $\widehat{\mathbb{C}}$ over $X$ but
the holomorphic motion $\phi$ cannot be extended to a holomorphic motion of $\widehat{\mathbb{C}}$ over $X$.
\end{enumerate}

\end{Mythm}
\begin{Rem}
Let $E'$ be any subset of $E$ consisting of four points. Then, for the holomorphic motion $\phi$ and the Riemann surface $X$ in the second statement, the monodromy of $\phi |{E'}$, which is the trace monodromy for $E'$, is trivial. 
Therefore, the holomorphic motion $\phi$ gives a negative answer to the question in \cite{BJMS2012} for the trace monodromy.
\end{Rem}
We define the group equivariance of holomorphic motions.
\begin{Def}
\label{dfn:G-equivariance}
	 Let $G$ be a subgroup of $\textrm{M\"ob}(\mathbb{C})$, the set of M\"obius transformations. 
Suppose that a closed set $E$ of $\widehat{\mathbb{C}}$ is $G$-invariant, that is, $GE=E$. 
A holomorphic motion $\phi : X\times E \to \widehat{\mathbb{C}}$ is called $G$-\emph{equivariant} if for each $p\in X$ there exists an isomorphism $\theta_{p}$ from $G$ to $\textrm{M\"ob}(\mathbb{C})$ such that 
\begin{equation}
	\label{eqn:equivariance}
	\phi (p, g(z))=\theta_{p}(g)(\phi (p, z))
\end{equation}
holds for any $(p, z, g)\in X\times E\times G$.
\end{Def}
We may also show an equivariant version of Theorem \ref{thm:mythm1}.
\begin{Mythm}
\label{thm:mythm2}
	Let $\phi : X\times E\to \widehat{\mathbb{C}}$ be a $G$-equivariant holomorphic motion of $E$ over a Riemann surface $X$. 
	Then there exists a $G$-equivariant holomorphic motion of $\widehat{\mathbb{C}}$ over $X$ if and only if the monodromy of $\phi$ is trivial.
	
\end{Mythm}

\section{Fundamental Notions}

\subsection{Teichm\"uller space and reduced Teichm\"uller space}
\label{subsection:Teich}
First of all, we recall the definition of Teichm\"uller space $Teich(S)$ of a hyperbolic Riemann surface $S$ (see \cite{Imayosh1992} for further details). 

Let $\Gamma$ be a Fuchsian group acting on the unit disk $\Delta$ which represents $S$. 
A quasiconformal self-map $\varphi$ of $\Delta$ is called $\Gamma$-{\em compatible} if $\varphi\circ \gamma\circ\varphi^{-1}$ is a M\"obius transformation for every $\gamma\in \Gamma$. 
Two $\Gamma$-compatible quasiconformal maps $\varphi_{1}, \varphi_{2}$ are said to be equivalent if there exists a M\"obius transformation $g$ from $\Delta$ to itself such that $g\circ\varphi_{1}=\varphi_{2}$ on $\partial\Delta$.
The Teichm\"uller space $Teich(S)$ of $S$ is defined as the set of equivalence classes $[\varphi]_{\Gamma}$ of $\Gamma$-compatible quasiconformal maps $\varphi$.

The Teichm\"uller space $Teich(S)$ is understood in terms of Beltrami coefficients.
For a measurable set $U\subset\mathbb{C} $ we denote by $M(U)$ the space of Beltrami coefficients on $U$, that is, the space of bounded measurable functions $\mu$ on $U$ with $||\mu||_{\infty}<1$. 
For each $\mu\in M(\Delta)$, we have a quasiconformal self-map $\varphi^{\mu}$ of $\Delta$ satisfying the Beltrami equation:
\begin{equation*}
	\overline \partial \varphi^{\mu}= \mu \partial \varphi^{\mu} 
\end{equation*}
almost everywhere in $\Delta$; $\mu$ is called the Beltrami coefficient of $\varphi^{\mu}$. Then $\varphi^{\mu}$ is $\Gamma$-compatible if and only if
\begin{equation}
\label{eq:compatibleBelt}
	\mu (\gamma (z))\frac{\overline{\gamma '(z)}}{\gamma '(z)}=\mu (z) \quad (a.e.)
\end{equation}
holds for every $\gamma\in \Gamma$ and for almost all $z$ in $\Delta$. 
We denote by $\textrm{Belt}(\Gamma)$ the set of Beltrami coefficients $\mu$ satisfying (\ref{eq:compatibleBelt}). 
We denote by $\pi_{\Gamma}$ the natural projection from $\textrm{Belt}(\Gamma)$ to $Teich(S)$ defined by $[{\mu}]_{\Gamma}$ for $\mu\in\textrm{Belt}(\Gamma)$.
It is known that the projection is surjective.
Therefore, the Teichm\"uller space $Teich(S)$ agrees with $\pi_{\Gamma}(\textrm{Belt}(\Gamma))$.

\medskip

A hyperbolic Riemann surface $S$ is called of type $(g, n)$ if it is a Riemann surface of genus $g$ with $n$ punctures. Then, the Teichm\"uller space $Teich(S)$ has another description.

Let $S'$ be a Riemann surface of the same type as $S$ and $f : S\to S'$ be a quasiconformal homeomorphism from $S$ to $S'$. 
We consider the pair $(S', f)$.
We say that two such pairs $(S_i, f_i)$ $(i=1, 2)$ are Teichm\"uller equivalent if there exists a conformal map $h : S_1\to S_2$ such that $h$ is homotopic to $f_2\circ f_1^{-1}$.
The Teichm\"uller equivalence class of $(S', f)$ is denoted by $[S', f]$.
The Teichm\"uller space $Teich(S)$ of $S$ is regarded as the set of all Teichm\"uller equivalence classes.
It is known that the Teichm\"uller space $Teich(S)$ admits a natural complex structure of dimension $3g-3+n$ if $S$ is of type $(g, n)$.

The Teichm\"uller space $Teich(S)$ is equipped with a complete distance called the Teichm\"uller distance $d_{T}$, which is defined by
\begin{equation*}
	d_{T}([S_1, f_1], [S_2, f_2])=\frac{1}{2}\inf_{g}\log K(g),
\end{equation*}
where the infimum is taken over all quasiconformal maps $g : S_1\to S_2$ homotopic to $f_2\circ f_1^{-1}$ and $K(g)$ is the maximal dilatation of $g$.
It is known that the Teichm\"uller distance $d_T$ is equal to the Kobayashi distance of the complex manifold $Teich(S)$.

The mapping class group $\textrm{Mod}(g, n)$ is the group of homotopy classes of all quasiconformal self-maps $\omega$ of $S$.
Let $\chi [\omega]$ be an element of $\textrm{Mod}(g, n)$ represented by $\omega : S\to S$.
Then, it acts on $Teich(S)$ by
\begin{equation*}
	\chi [\omega]([S', f])=[S', f\circ \omega^{-1}] \quad ([S', f]\in Teich(S)).
\end{equation*} 
It is not hard to see that the action is well-defined and every $\chi [\omega]$ is an isometry with respect to the Teichm\"uller distance.
In fact, the mapping class group $\textrm{Mod}(g, n)$ is the group of biholomorphic automorphisms of $Teich(S)$.
\medskip

{\textbf{Reduced Teichm\"uller space.}}

A Riemann surface $S$ is called of type $(g; n, m)$ if it is of genus $g$ with $n$ punctures and $m$ boundaries. 
We assume that the surface is hyperbolic so that its universal covering is conformally equivalent to the unit disk $\Delta$.
We may also assume that every boundary component is a smooth Jordan curve.

Let $S_0$ be a Riemann surface of type $(g; n, m)$ with $m\geq 1$ and $\Gamma_0$ be a Fuchsian group acting on $\Delta$ which represents $S_0$.
Then the Fuchsian group $\Gamma_0$ is of the second kind, that is, the limit set $\Lambda (\Gamma_0)$ is a Cantor set on $\partial\Delta$.
We say that two $\Gamma_0$-compatible quasiconformal maps $\varphi_1, \varphi_2$ are {\em R-equivalent} if there exists a M\"obius transformation $g$ from $\Delta$ to itself such that $g\circ\varphi_1=\varphi_2$ on $\Lambda(\Gamma_0)$.
The reduced Teichm\"uller space $Teich^{\#}(S_0)$ of $S_0$ is defined as the set of all R-equivalence classes $[\varphi]_{\Gamma_0}^{\#}$ of $\Gamma_0$-compatible quasiconformal self-maps $\varphi$ of $\Delta$.

For $[\varphi]_{\Gamma_0}^{\#}\in Teich^{\#}(S_0)$, the quasiconformal self-map $\varphi$ of $\Delta$ can be symmetrically extended to a self-quasiconformal map $\widehat{\varphi}$ of the Riemann sphere $\widehat{\mathbb{C}}$.
We see that $\Gamma_{\varphi}:=\widehat{\varphi}\Gamma_0\widehat{\varphi}^{-1}$ is also a Fuchsian group of the second kind.
It acts properly discontinuously on $\widehat{\mathbb{C}}\setminus\Lambda (\Gamma_{\varphi})$.
Hence, it gives a quasiconformal map from $\left (\widehat{\mathbb{C}}\setminus\Lambda(\Gamma_0)\right )/\Gamma_0$ to
$\left ( \widehat{\mathbb{C}}\setminus\Lambda(\Gamma_{\varphi})/\Gamma_{\varphi}\right )$. 
Since $\widehat{S}_0:= \left (\widehat{\mathbb{C}}\setminus\Lambda(\Gamma_0)\right )/\Gamma_0$ and $\widehat{S}_{\varphi}:=\left ( \widehat{\mathbb{C}}\setminus\Lambda(\Gamma_{\varphi})/\Gamma_{\varphi}\right )$ 
are doubles of $S_0$ and $S_{\varphi}:=\Delta/\Gamma_{\varphi}$ respectively, it defines an element $\mu_{\varphi}$ of $\textrm{Belt}(\widehat{\Gamma}_0)$, where $\widehat{\Gamma}_0$ is a Fuchsian group representing $\widehat{S}_0$.
It is known that if $\varphi$ and $\psi$ are R-equivalent, then $\mu_{\varphi}$ and $\mu_{\psi}$ gives the same point of $Teich(\widehat{S}_0)$.
Thus, we have a map $\widehat{\Pi}$ from $Teich^{\#}(S_0)$ to $Teich(\widehat{S}_0)$. 
In fact, the map is real analytic.

\medskip

{\bf Teichm\"uller space of a closed set.}

Let $E$ be a closed set in $\widehat{\mathbb{C}}$ containing $0, 1, \infty$. 
It is well known that for each $\mu\in M(\mathbb{C})$, there exists a unique quasiconformal self map $w^{\mu}$ of $\widehat{\mathbb{C}}$ such that it fixes $0, 1, \infty$ and satisfies the Beltrami equation
$$
\overline\partial w^{\mu}=\mu  \partial w^{\mu}
$$
almost everywhere in $\mathbb{C}$.
The map $w^{\mu}$ is called the \emph{normalized quasiconformal map} for $\mu$.

Two Beltrami coefficients $\mu, \nu \in M(\mathbb{C})$ are said to be $E$-\emph{equivalent} if $(w^{\mu})^{-1}\circ w^{\nu}$ are homotopic to the identity $\textrm{rel }E$.
We denote by $[\mu]_{E}$ the equivalence class of $\mu$.
The Teichm\"uller space $T(E)$ of the closed set $E$ is the set of all equivalence classes $[\mu]_{E}$ of $\mu\in M(\mathbb{C})$.
Obviously, $T(E)=Teich(\widehat{\mathbb{C}}\setminus E)$ if $E$ is a finite set.
We denote by $P_{E}$ the quotient map of $M(\mathbb{C})$ onto $T(E)$.

Since $E$ is a closed set of $\widehat{\mathbb{C}}$, the complement $E^c$ of $E$ is a disjoint union of domains $X_{i}$ $(i\in \mathbb{N})$ on $\widehat{\mathbb{C}}$. 
Each $X_i$ is a hyperbolic Riemann surface; we may define Teichm\"uller space $T(E^c)$ as the product Teichm\"uller space $\amalg_{i\in \mathbb{N}}T(X_{i})$ with product metric.
It is the set of all $(p_i)_{i\in \mathbb{N}}$ $(p_i\in T(X_i))$ such that the Teichm\"uller distances between the basepoint $X_i$ and $p_i$ is less than a constant independent of $i\in \mathbb{N}$.

A Beltrami coefficient on $X_i$ gives a point in $T(X_i)$; it is computed that the Teichm\"uller distance between two points determined by 0 and $\mu\in M(X_i)$ is not greater than $\frac{1}{2} \log (1+\|\mu\|_{\infty})(1-\|\mu\|_{\infty})^{-1}$. 
Hence, considering the restriction $\mu|E^c$ for $\mu\in M(\mathbb{C})$, we have a map $P_T$ from $M(\mathbb{C})$ to $T(E^c)$ which sends $\mu\in  M(\mathbb C)$ to the product of $p_i\in T(X_i)$ defined by $\mu |X_j$ $(i\in \mathbb N)$.
We define a map $\tilde{P}_E$ from $M(\mathbb{C})$ to $T(E^c)\times M(E)$ by
\begin{equation}
	\tilde{P}_{E}(\mu)=(P_T(\mu), \mu|E).
\end{equation}
It is easy to see that $\tilde{P}_{E}$ is surjective. 
Moreover, it is known that the map is a holomorphic split submersion.
\begin{Pro}[Lieb's isomorphism theorem. cf.~\cite{Earle2000}]
\label{Pro:Lieb}
	There exists a well-defined bijective map $\theta : T(E)\to T(E^c)\times M(E)$ such that $\theta\circ P_{E}=\tilde{P}_{E}$. 
	The Teichm\"uller space $T(E)$ admits a unique complex structure so that $P_E$ is a holomorphic split submersion and $\theta$ is biholomorphic.
\end{Pro}

If $\mu, \nu\in M(\mathbb{C})$ are $E$-equivalent, then $w^{\mu}(z)=w^{\nu}(z)$ for all $z\in E$. 
Thus, we have a well-defined map of $T(E)$ to $\widehat{\mathbb{C}}$ by $T(E)\ni [\mu]_{E}\mapsto w^{\mu}(z)\in \widehat{\mathbb{C}}$.
It is known that $\Psi ([\mu]_{E}, z):=w^{\mu}(z)$ is a holomorphic motion of $E$ over $T(E)$. 
Furthermore, the holomorphic motion $\Psi$ is \emph{universal} in the following sense (cf.~\cite{Earle2000}).

\begin{Pro}
\label{Pro:universal}
	Let $V$ a simply connected complex Banach manifold and $\phi : V\times E\to \widehat{\mathbb{C}}$ a holomorphic motion of $E$ over $V$. 
	Then, there exists a holomorphic map $f : V\to T(E)$ such that $\phi (p, w)=\Psi (f(p), w)$ hold for all $(p, w)\in V\times E$.
\end{Pro}
\subsection{Douady-Earle extension}
\label{subsection:Douady-Earle}
Let $h$ be an orientation preserving homeomorphism of the unit circle to itself. Douady and Earle (\cite{Douady1986}) found a canonical extension $E(h)$ of $h$ to the unit disk $\Delta$, which they call the \emph{barycentric extension}; it is now also called \emph{Douady-Earle extension} (cf.\ see also \cite{Lecko1988})
The extension $E(h)$ has the following significant properties:
\begin{enumerate}
  \item $E(h)$ is real analytic and it is {\emph{conformally natural}}, that is, for conformal automorphisms $f, g$ of $\Delta$,
\begin{equation}
\label{eq:confNatural}
	E(f\circ h\circ g)=f\circ E(h)\circ g
\end{equation}
holds in $\Delta$;
  \item if $h$ is the boundary function of a quasiconformal self-map of $\Delta$, then $E(h)$ is also quasiconformal.
\end{enumerate}

Let $\Gamma$ be a Fuchsian group acting on $\Delta$ which represents a hyperbolic Riemann surface $S$. For each $p=[\varphi]_{\Gamma}\in Teich(S)$, there exists an isomorphism $\theta_{p}$ of $\Gamma$ to $\textrm{M\"ob}(\Delta)$, the space of M\"obius transformations preserving $\Delta$ such that
\begin{equation*}
	\varphi\circ\gamma=\theta_{p}(\gamma)\circ \varphi
\end{equation*}
holds on $\partial \Delta$ for any $\gamma\in\Gamma$. Hence, from (\ref{eq:confNatural}) we have
\begin{equation*}
	E(\varphi)\circ\gamma=\theta_{p}(\gamma)\circ E(\varphi) \quad  \textrm{ in }\Delta ,  
\end{equation*}
and $\sigma_S : Teich(S)\ni p \mapsto E(\varphi)_{\bar{z}}/E(\varphi)_{z}\in \textrm{Belt}(\Gamma)\subset M(\Delta )$ is a well-defined map of $Teich(S)$ to $\textrm{Belt}(\Gamma)$. 
The map $\sigma_R$ has the following properties:
\begin{enumerate}
  \item It is a section for $\pi_{\Gamma} : \textrm{Belt}(\Gamma)\to Teich(S)$, that is, $\pi_{\Gamma}\circ \sigma_{S}=id$ and it is real analytic;
  \item for each $p\in Teich(S)$, $\sigma_{S}(p)$ is real analytic as a function on $\Delta$.
\end{enumerate}

\medskip

Using $\sigma_{S}$, we give a section of $P_{E}$ for the Teichm\"uller space $T(E)$ of a closed set $E$. 
In Proposition \ref{Pro:DEsection} below, we see that the section is real analytic. The further discussions are given in a paper of Earle-Mitra \cite{EarlePre} with more complete details. 
We will give a brief discussion for the convenience of the reader.

\medskip

From Proposition \ref{Pro:Lieb}, $T(E)$ is biholomorphic to $T(E^c)\times M(E)$.
We shall give a section for $\tilde{P}_{E} : M(\mathbb{C}) \to T(E^c)\times M(E)$.

Take a point $(p, \nu)$ in $T(E^c)\times M(E)$. 
The point $p=(p_i)_{i\in \mathbb{N}}$ is a point of the product Teichm\"uller space $\amalg_{i\in \mathbb{N}}T(X_{i})$, where $X_{i}$ are connected components of $E^c$ and $p_{i}\in T(X_i)$ $(i=1, 2, \dots )$.
Thus, we have $\sigma_{X_i}(p_i)\in \textrm{Belt}(\Gamma_i)$ for a Fuchsian group for $X_i$.
Projecting $\sigma_{X_i}(p_i)$ to the planar domain $X_i$, we have a measurable function on $X_i$, which we denote by $\widehat{\sigma}(p_i)$.

Since $\sigma_{X_i}(p_i)$ is real analytic on $\Delta$, so is $\widehat{\sigma}(p_i)$ on $X_i$.
Moreover, a property of Douady-Earle extension shows that there exists a constant $k_p \in [0, 1)$ depending only on $p$ such that $\|\widehat{\sigma}(p_i)\|_{\infty}\leq k_p$. 
Therefore, we may define a map $s_{E}$ from $T(E^c)\times M(E)$ to $M(\mathbb{C})$ by
\begin{equation}
	s_{E}(p, \nu)(z)=
	\begin{cases}
		\widehat{\sigma}(p_i)(z), & \qquad (z\in X_i; i\in \mathbb{N}), \\
		\nu (z), & \qquad (z\in E).
	\end{cases}
\end{equation}
Then we have (see also \cite{Earle2000})
\begin{Pro}
	\label{Pro:DEsection}
	The map $s_{E}$ is a section for $\tilde{P}_{E} : M(\mathbb{C})\to T(E^c)\times M(E)$. Furthermore,
	\begin{enumerate}
  \item $s_{E} : T(E^c)\times M(E)\to M(\mathbb{C})$ is real analytic;
  \item for each $(p, \nu)\in T(E^c)\times M(E)$, $s_{E}(p, \nu)$ is a real analytic function on $E^c$.
\end{enumerate}

\end{Pro}
\begin{Def}
	The $s_{E}$ defined as above is called the Douady-Earle section on $T(E)$.
\end{Def}

\medskip

{\bf A property in the reduced Teichm\"uller space.}

Here, we shall note a property of $\sigma_{S}$ related to reduced Teichm\"uller space.

Let $S_0$ be a hyperbolic Riemann surface of type $(g; n, m)$ with $m\geq 1$ and $\widehat{S}_0$ be the double of $S_0$ with respect to $\partial S_0$.
Then, there exists an anti-conformal involution $J_{S_0} : \widehat{S}_0\to \widehat{S}_0$ which keeps every boundary point fixed.
Let $f : S_0\to S$ be a quasiconformal map. Since $S$ is also a Riemann surface of the same type, we may consider the double $\widehat{S}$ of $S$ with respect to $\partial S$ and the anti-conformal involution $J_{S} : \widehat{S}\to \widehat{S}$ for $\widehat{S}$. 
We extend $f$ to a map $\hat{f} : \widehat{S}_0\to \widehat{S}$ symmetrically by
\begin{equation}
\hat{f}(z)=
	\begin{cases}
	f(z), \quad & z\in {S}_0 \\
	J_{S}(f(J_{S_0}(z))), \quad & z\in \widehat{S}_0\setminus S_0.
	\end{cases}
\end{equation}
Since those surfaces are bounded by smooth Jordan curves, $\hat{f}$ is a quasiconformal map on $\widehat{S}_0$.
In fact, we have $\widehat\Pi ([S, f])=[\widehat S, \hat f]\in Teich(\widehat S_0)$.

From the definition of $\hat f$, we have
$$
J_{S}\circ \hat{f}=\hat f\circ J_{S_0}.
$$
Hence, we may assume that there is a lift $F :\Delta\to\Delta$ of $\hat f$ such that
\begin{equation*}
	j\circ F=F\circ j,
\end{equation*}
where $j(z)=\bar z$. Then the boundary function $\varphi_F$ of $F$ also satisfies
\begin{equation}
\label{eqn:symmetric}
	j\circ \varphi_F=\varphi_F\circ j.
\end{equation}
From (\ref{eqn:symmetric}) and the conformal naturality (\ref{eq:confNatural}), we have
\begin{equation}
	j\circ E(\varphi_F)=E(\varphi_F)\circ j
\end{equation}
on $\Delta$.
The map $E(\varphi_F)$ is projected to a real analytic quasiconformal map $e(f)$ from $\widehat{S}_0$ to $\widehat{S}$, and $e(f)$ is symmetric, namely,
\begin{equation}
	J_{S}\circ e(f)=e(f)\circ J_{S_0}
\end{equation}
holds on $\widehat{S}_0$.
Therefore, $e(f)(\partial S_0)=\partial S$ and we see that $e(f)(S_0)=S$ because $e(f)$ is orientation preserving.
The symmetric quasiconformal map $e(f)$ determines the same point as $\hat f$ in $T(\widehat{S}_0)$.
Hence $e(f)|S_0 : S_0\to S$ gives the same point as $f$ in $Teich^{\#}(S_0)$.

Summing up the above arguments, we conclude the following proposition which is used in Step 2 of \S 5:
\begin{Pro}
	\label{Pro:symmetricQC}
	Let $S_0$ be a Riemann surface of type $(g; n, m)$ with $m\geq 1$ and $f : S_0\to S$ be a quasiconformal homeomorphism. Then, there exists a real analytic symmetric quasiconformal homeomorphism $e(f) : \widehat{S}_0 \to \widehat{S}$ such that $e(f)|S_0$ determines the same point as $f$ in $Teich^{\#}(S_0)$.
\end{Pro}

\section{Extending holomorphic motions as quasiconformal motions}
In this section, we shall view the extension problem for holomorphic motions in a general setting. 
\begin{Pro}
\label{Pro:tameQC}
	Let $M$ be a connected complex manifold and $\phi : M\times E\to \widehat{\mathbb{C}}$ a holomorphic motion of a closed set $E$ over $M$. 
	Suppose that the monodromy of $\phi$ is trivial. Then there exists a map $\Phi : M\to M(\mathbb{C})$ with the following properties:
	\begin{enumerate}
  \item The map $\Phi$ is real analytic;
  \item for each $p\in M$, $\Phi (p)\in M(\mathbb{C})$ is real analytic in $\mathbb{C}\setminus E$;
  \item $w^{\Phi (p)}(z)=\phi (p, z)$ for every $(p, z)\in M\times E$, where $w^{\mu}$ is the normalized quasiconformal self-map of $\widehat{\mathbb{C}}$ whose Beltrami coefficient is $\mu$.
\end{enumerate}
\begin{proof}
Let $\tilde{M}$ be the holomorphic universal covering of $M$ and $\Gamma$ a cover transformation group of $\tilde{M}$.
Then, the canonical projection $\pi : \tilde{M}\to M=\tilde{M}/\Gamma$ defines a holomorphic motion $\tilde{\phi}$ of $E$ over $\tilde{M}$ by
\begin{equation}
\label{eqn:liftofHolMot}
	\tilde{\phi}(\tilde p, z)=\phi(\pi (\tilde p), z), \qquad (\tilde p, z)\in \tilde{M}\times E.
\end{equation}

Since $\tilde{M}$ is simply connected, it follows from Proposition \ref{Pro:universal} that there is a holomorphic map $f : \tilde{M}\to T(E)$ such that
\begin{equation}
\label{eqn:Univ.ToTeich}
	\tilde{\phi}(\tilde p, z)=\Psi (f(\tilde p), z), \quad (\tilde p, z)\in \tilde{M}\times E.
\end{equation} 
Moreover, it follows from the triviality of the monodromy of $\phi$ that $f\circ \gamma=f$ for any $\gamma\in \Gamma$.
Thus, we have a well-defined holomorphic map $F : M\to T(E)$ with $F\circ \pi=f$.
Put $\Phi = s_{E}\circ F$, where $s_{E}$ is the Douady-Earle section given in Proposition \ref{Pro:DEsection}.
We verify that the map $\Phi$ has the desired properties.
Indeed, the statements (1) and (2) are direct consequences from Proposition \ref{Pro:DEsection}. We show (3).

For any $(p, z)\in M\times E$, we take $\tilde p\in \tilde M$ such that $\pi (\tilde p)=p$.
It follows from (\ref{eqn:liftofHolMot}) and (\ref{eqn:Univ.ToTeich}) that 
\begin{eqnarray*}
	\phi (p, z) &=& \tilde \phi (\tilde p, z) = \Psi (f(\tilde p), z) = \Psi (F\circ \pi (\tilde p), z) \\
	&=& \Psi (F(p), z)=w^{s_E\circ F(p)}(z)=w^{\Phi (p)}(z).
\end{eqnarray*}
Thus, we have proved (3).
\end{proof}
\end{Pro}
\begin{Rem}
	For a given holomorphic motion $\phi : M\times E\to \widehat{\mathbb{C}}$ with trivial monodromy, the above $\Phi : M\to M(\mathbb{C})$ defines a \lq\lq motion\rq\rq  $\bar{\phi} : M\times \widehat{\mathbb{C}}\to \widehat{\mathbb{C}}$ of $\widehat{\mathbb{C}}$ over $M$ by
	\begin{equation}
	\label{eqn:qcMotion}
		\bar{\phi}(p, z)=w^{\Phi (p)}(z), \qquad (p, z)\in M\times \widehat{\mathbb{C}}.
	\end{equation}
	
It is not holomorphic because $\Phi$ is not necessarily holomorphic. 
However, it still makes a continuous family of quasiconformal maps parametrized by $M$. 
In fact, $\bar{\phi}$ is a quasiconformal motion in the sense of Sullivan and Thurston \cite{Sullivan1986}. 
\end{Rem}
\section{Reductions of Theorem \ref{thm:mythm1}}
The proof of Theorem \ref{thm:mythm1} is done by several steps in \S5 while it uses an idea similar to that in Slodkowski \cite{Slodkowski1995}. However, our methods in \S 5 are slightly different and they are simpler than the argument in \cite{Slodkowski1995}. 

In this section, we reduce the statement to a simpler one; first, we see that the argument is deduced from the case where the set $E$ is finite.
\begin{lemma}
\label{lemma:finiteE}
	Let $\phi : X\times E\to\widehat{\mathbb{C}}$ be a holomorphic motion of $E$ over a Riemann surface $X$. 
	Then, $\phi$ can be extended to a holomorphic motion of $\widehat{\mathbb{C}}$ over $X$ if and only if for any finite set $E'$ of $E$ the restricted holomorphic motion to $E'$ is extended to a holomorphic motion of $\widehat{\mathbb{C}}$ over $X$.
\end{lemma}
\begin{proof}
	\lq\lq only if\rq\rq statement is obvious; we give a proof of \lq\lq if"-part.
	In fact, one may find a proof in some literature (e.~g.~\cite{BJMS2012}). We will give the proof for the convenience of the reader.
	
	Let $\{0, 1, \infty\}\subset E_1\subset E_2\dots \subset E_n\dots \subset E$ be a sequence of finite subsets of $E$ with $E=\overline{\cup_{n=1}^{\infty}E_n}$. 
	From the assumption, there exist holomorphic motions $\phi_{n} : X\times \widehat{\mathbb{C}} \to\widehat{\mathbb{C}}$ of $\widehat{\mathbb{C}}$ $(n=1, 2, \dots, )$ which extend $\phi |{E_n}$. 
	We take a countable dense set $A:=\{a_{k}\}_{k=1}^{\infty}$ of $\widehat{\mathbb{C}}\setminus\{0, 1, \infty\}$ containing all $E_{n}$ $(n=1, 2, \dots)$.
	
	Now, we consider $\{\phi_{n}(\cdot, a_1)\}_{n=1}^{\infty}$, a family of holomorphic functions on $X$. 
	Since $\phi_{n}(X, a_{1})\subset\widehat{\mathbb{C}}\setminus\{0, 1, \infty\}$, the family is a normal family by Montel's theorem. 
	Hence, we may find a subsequence $\{n_{1j}\}_{j=1}^{\infty}$ of $\mathbb{N}$ such that $\{\phi_{n_{1j}}(\cdot, a_{1})\}_{j=1}^{\infty}$ normally converges to a holomorphic function on $X$.
	By the same reason, we may find a subsequence $\{n_{2j}\}_{j=1}^{\infty}$ of $\{n_{1j}\}_{j=1}^{\infty}$ such that $\{\phi_{n_{2j}}(\cdot, a_{2})\}_{j=1}^{\infty}$ normally converges to a holomorphic function on $X$.
	Using the same argument and the diagonal argument, we have a sequence $\{n_{k}\}_{k=1}^{\infty}\subset \mathbb{N}$ such that $\{\phi_{n_k}(\cdot, a)\}_{k=1}^{\infty}$ converges to a holomorphic function $F(\cdot, a)$ on $X$ for each $a\in A$. 
	
	Since $\phi_{n_k}$ is an extension of $\phi$, we have
	$$
	F(p, z)=\lim_{k\to \infty}\phi_{n_k}(p, z)=\phi (p, z)
	$$
	for each $p\in X$ and $z\in E_n$ and $F(p_{0}, z)=z$ for all $z\in \widehat{\mathbb{C}}$.
	
	Next, we show the injectivity of $F(p, z)$ for $z\in\widehat{\mathbb C}$.
	We see that the map $\widehat{\mathbb{C}}\ni z\mapsto \phi_{n}(p, z)$ gives a normalized quasiconformal map for each $p\in X$. 
	Moreover, for each compact subset $V$ of $X$, there exists a constant $K>1$ depending only on $V$ such that $\phi_{n}(p, \cdot)$ is a normalized $K$-quasiconformal map on $\widehat{\mathbb{C}}$ for any $p\in V$ (cf. \cite{Bers1986a} Theorem 2). 
	Thus, $\{\phi_{n}(p, \cdot)\mid p\in V \}_{n=1}^{\infty}$ is a normal family of quasiconformal maps on $\widehat{\mathbb{C}}$. 
	
	We have already taken the sequence $\{n_{k}\}_{k=1}^{\infty}$ above such that $\{\phi_{n_{k}}(\cdot, \cdot)\}_{k=1}^{\infty}$ converges to $F(\cdot, \cdot)$ on $X\times A$. 
	Since $\overline{A}=\widehat{\mathbb{C}}$, it converges also on $X\times \widehat{\mathbb{C}}$ and $F(p, \cdot)$ is a quasiconformal map. 
	In particular, $F(p, z)\not=F(p, z')$ if $z\not=z'$.  
	
	Above all, we have a holomorphic motion $F : X\times \widehat{\mathbb C} \to \widehat{\mathbb{C}}$ of $\widehat{\mathbb C}$ over $X$ which extends $\phi$ as desired.
\end{proof}
We may assume that the set $E$ is finite. In the next step, we see that the Riemann surface could be of analytically finite and the holomorphic motion is defined over the closure of the Riemann surface.
\begin{lemma}
	Let $\phi : X\times E\to \widehat{\mathbb{C}}$ be a holomorphic motion of a finite set $E$ over a Riemann surface $X$. 
	Then, $\phi$ can be extended to a holomorphic motion of $\widehat{\mathbb{C}}$ over $X$ if and only if for any subsurface $X'$ of $X$ the restricted holomorphic motion to $X'$ can be extended to a holomorphic motion of $\widehat{\mathbb{C}}$ over $X'$.
\end{lemma}
\begin{proof}
	Again, \lq\lq only if\rq\rq statement is obvious; we give a proof of \lq\lq if\rq\rq-part.
	
	Let $x_{0}\in X_{1}\subset X_{2}\subset \dots , X_{n}\subset X_{n+1}\subset \dots$ be a sequence of domains of $X$ satisfying the following conditions.
	\begin{enumerate}
  \item each domain $X_{n}$ is bounded by a finite number of analytic Jordan curves in $X$;
  \item $\overline{X_{n}}:=X_{n}\cup \partial X_{n}$ is compact in $X$ $(n=1, 2, \dots )$;
  \item every connected component of $X\setminus X_{n}$ is non-compact in $X$ $(n=1, 2, \dots )$;
  \item $X=\cup_{n=1}^{\infty}X_{n}$.
\end{enumerate}
Here, a simple curve $\gamma$ on a Riemann surface $X$ is called {\em{analytic}} if for each point $p\in \gamma$ there exist a neighborhood $U_{x}$ of $p$ and a conformal map $\varphi : U_{p}\to \mathbb{C}$ such that $\varphi (U_{p}\cap \gamma)\subset \mathbb{R}$.
Sometimes, $\{X_{n}\}_{n=1}^{\infty}$ is called a \emph{regular exhaustion} of $X$. The existence of a regular exhaustion of an open Riemann surface is well known (cf. \cite{Ahlfors-Sario}).

Let $\psi_{n} : X_{n}\times E\to \widehat{\mathbb{C}}$ be a holomorphic motion restricted to $X_{n}$, namely, $\psi_{n}(p, z)=\phi (p, z)$ for $(p, z)\in X_{n}\times E$.
From our assumption, there exist holomorphic motions $\widehat{\psi}_{n} : X_{n}\times \widehat{\mathbb{C}}\to \widehat{\mathbb{C}}$ of $\widehat{\mathbb{C}}$ over $X_{n}$ which extends $\psi_{n}$. 
For each $z\in \widehat{\mathbb{C}}\setminus\{0, 1, \infty\}$, we have $\widehat{\psi}_{n}(X_{n}, z)\subset \widehat{\mathbb{C}}\setminus \{0, 1, \infty\}$ $(n=1, 2, \dots )$. 
Hence, by Montel's theorem the family $\{\widehat{\psi}_n(\cdot , z)|D\}_{n=k}^{\infty}$ is normal in any domain $D$ contained in $X_k$.

Take a countable dense subset $A$ in $\widehat{\mathbb{C}}\setminus E$ and repeat the same argument as in the proof of Lemma \ref{lemma:finiteE}. 
Then, as a limit of (a subsequence of) $\{\widehat{\psi}_{n}(\cdot, z)\}_{n=1}^{\infty}$ $(z\in A)$, we have a holomorphic motion of $\widehat{\mathbb{C}}$ over $X$ which extends $\phi$.
\end{proof}
We say that a Riemann surface $X$ is \emph{compact bordered} if it is of finite genus and bounded by a finite number of analytic Jordan curves. 
Hence, to prove Theorem \ref{thm:mythm1}, it suffices to show that a holomorphic motion $\phi : X\times E\to \widehat{\mathbb{C}}$ of a finite set $E=\{0, 1, \infty, z_1, \dots, z_n\}$ over a compact bordered Riemann surface $X$ is extended to a holomorphic motion of $\widehat{\mathbb{C}}$ over $X$. 
Furthermore, we may assume that the holomorphic motion $\phi$ is a holomorphic motion over a neighborhood of $\overline{X}$.

For such a holomorphic motion $\phi : X\times E\to \widehat{\mathbb{C}}$, we set
\begin{align}
\label{eqn:radius}
	R=\max \{|\phi (p, z_{j})|\mid p\in \overline{X}, 1\le j\le n\}+1, \\
	r=\frac{1}{3} \min \{|\phi (p, z_{j})|\mid x\in \overline{X}, 1\le j\le n\}, \notag
\end{align}
and
$$
\Delta_{R}^{c}=\{z\in \mathbb{C} \mid |z|\geq R\}, \quad \Delta_{r}=\{w\in\mathbb{C}\mid |z|<r\}.
$$
Note that $0<r<3r<R$. 
Then, we define a holomorphic motion $\tilde{\phi}$ of $E\cup\overline{\Delta}_{r}\cup\Delta_{R}^{c}$ over $X$ by
\begin{equation}
	\tilde{\phi}(p, z)=
	\begin{cases}
		\phi (p, z), & \quad (p, z)\in X\times E, \\
		z, & \quad (p, z)\in X\times (\overline{\Delta}_{r}\cup\Delta_{R}^{c}).
	\end{cases}	
\end{equation}
Obviously, the monodromy of $\tilde{\phi}$ is trivial if that of $\phi$ is trivial. 
Hence, Theorem \ref{thm:mythm1} is deduced from the following theorem.
\begin{thm}
\label{them:reduction}
	Let $E$ be a finite set containing $0, 1$ and $\infty$ and $X$ be a compact bordered Riemann surface with basepoint $p_0$ which is not simply connected.
	Let $\phi : \overline{X}\times (E\cup\overline{\Delta}_{r}\cup\Delta_{R}^{c})\to \widehat{\mathbb{C}}$ be a holomorphic motion of $E\cup\overline{\Delta}_{r}\cup\Delta_{R}^{c}$ over $\overline{X}$, the closure of $X$.
	Suppose that $\phi (p, z)=z$ for $(p, z)\in \overline{X}\times (\overline{\Delta}_{r}\cup\Delta_{R}^{c})$ and the monodromy of $\phi$ is trivial.
	Then, there exists a holomorphic motion $\widehat{\phi}$ of $\widehat{\mathbb{C}}$ over $\overline{X}$ which extends $\phi$.
\end{thm}
\section{Proof of Theorem \ref{them:reduction}}
The proof of Theorem \ref{them:reduction} is done by several steps.
We put $E=\{0, 1, \infty, z_1, \dots , z_n\}$ and $R, r$ are defined by (\ref{eqn:radius}).

\medskip

\textbf{Step 1: Construction of a simply connected Riemann surface.} 
　
We take a set of analytic simple closed curves and arcs in $X$, say $\{\alpha_{1}, \alpha_{2}, \dots , \alpha_{k}\}$, having the following properties (see {\sc{Figure}} \ref{Fig1}).

\begin{enumerate}
\renewcommand{\labelenumi}{(\roman{enumi})}
  \item $\sharp (\alpha_{i}\cap\alpha_{j})\leq 1$ $(i\not=j)$ and $\sharp (\alpha_{i}\cap\partial X)\leq 2$ $(i, j=1, 2, \dots , k)$;
  \item if $\alpha_{i}\cap\alpha_{k}\not=\emptyset$ or $\alpha_i\cap\partial X \not=\emptyset$, then both curves cross each other perpendicularly;
  \item $X' := X\setminus \cup_{i=1}^{k}\alpha_{i}$ is connected and simply connected;
  \item $X'\ni p_0$.
\end{enumerate}

\begin{figure}[H]
\centering
\includegraphics[width=8cm]{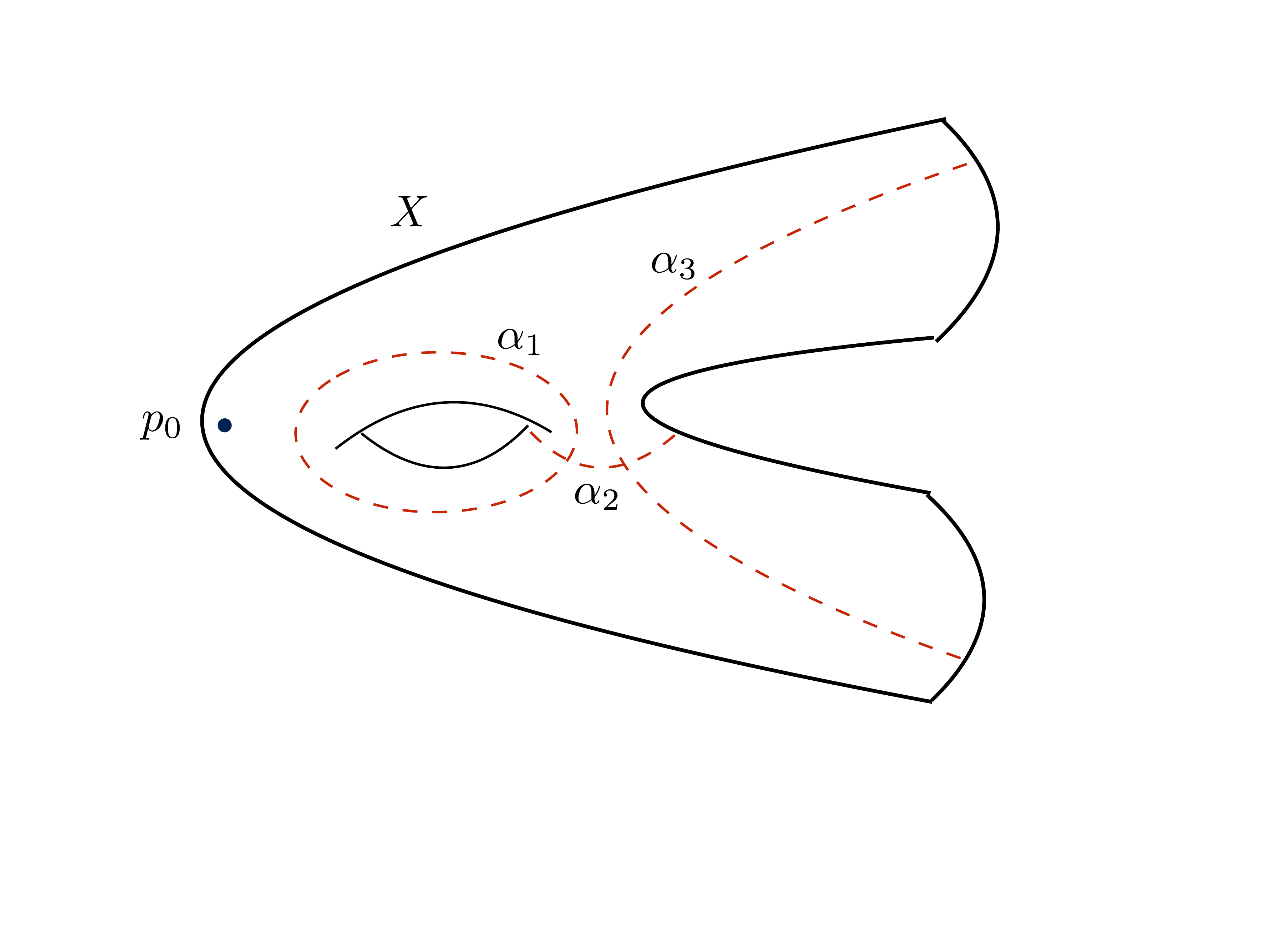}
\caption{}
\label{Fig1}
\end{figure}

The existence of such a system of curves and arcs is known; e.~g., use Green lines on $X$ of Green's function with pole at $p_0$ (cf.\ \cite{Sario1970}).
We denote by $\{q_1, q_2, \dots , q_m\}$ the set of all intersection points $\alpha_{i}\cap\alpha_{j}$ and $\alpha_{i}\cap\partial X$ $(i, j =1, 2, \dots , k)$.
Since $X'$ is a simply connected Riemann surface, we have a Riemann map $\varphi : \Delta \to X'$ with $\varphi (0)=p_0$.
The map $\varphi$ can be continuously extended to $\partial \Delta$. 
Moreover, since $\partial X'$ is piecewise analytic, $\varphi$ has an analytic continuation across $\partial\Delta$  except
$\{\xi_1, \xi_2, \dots , \xi_{m'}\}:= \cup_{i=1}^{m} \varphi^{-1}(q_{i})$, and $\varphi|\partial\Delta$ is a two to one map except for $\xi_{i}$ $(i=1, 2, \dots , m')$.
We distinguish points on $\partial X'$ according to images of $\varphi$.
Namely, if $\partial X'\ni q=\varphi (\xi_1)=\varphi(\xi_2)$ for $\xi_1\not=\xi_2$, we consider $q_{\xi_i}\in \partial X'$ as $\lim_{\Delta\ni \lambda \to \xi_i}\varphi(\lambda )$ $(i=1, 2)$ and they are different ({\sc{Figure}} \ref{Fig:disitinct}).
Hence, $\varphi$ becomes a homeomorphism from $\overline{\Delta}$ onto $\overline{X'}=X'\cup\partial X'$.

\begin{figure}
\centering
\includegraphics[scale=.65, bb=100 200 700 475]{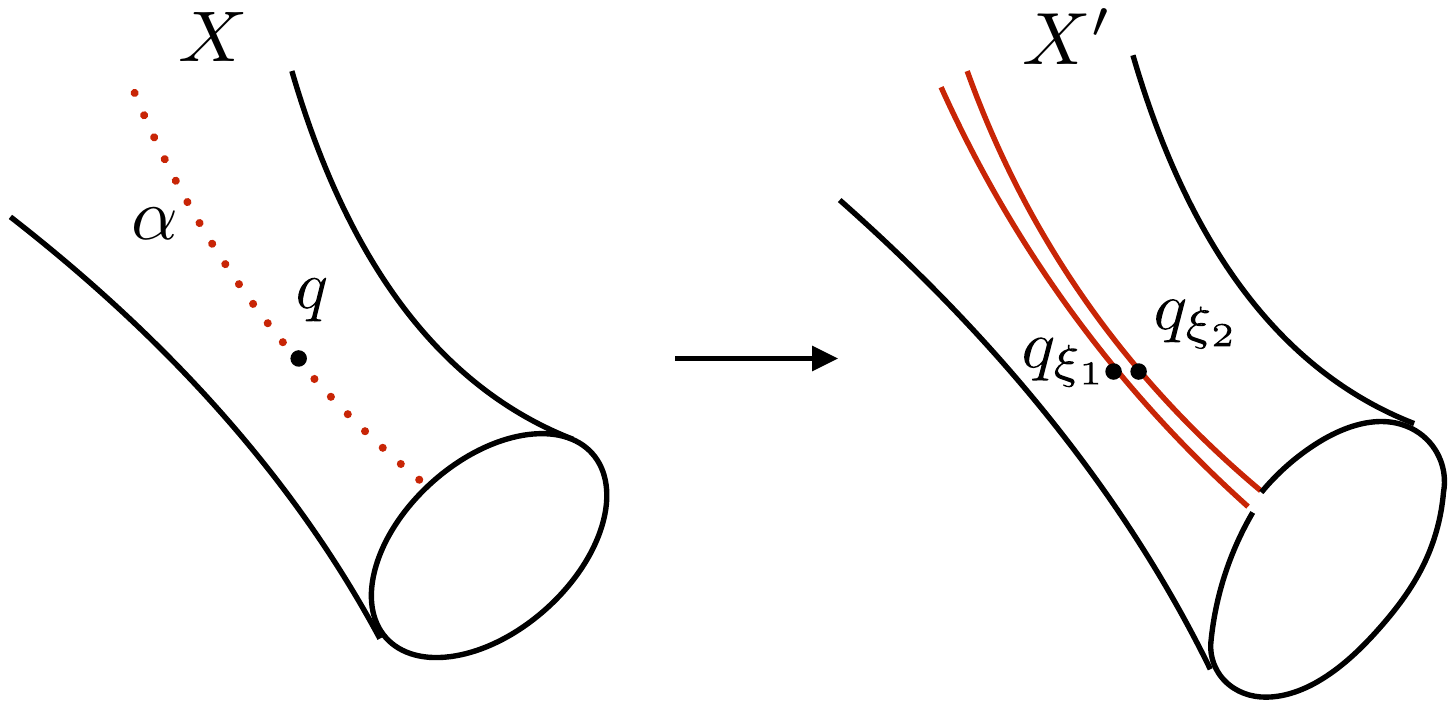}
\caption{}
\label{Fig:disitinct}
\end{figure}

Since the angle of $\partial X'$ at each $q_i$ is $\pi/2$, we have
\begin{equation}
  \label{eqn:Holder}
  |\varphi (\lambda )-\varphi (\xi_{j})|=O(|\lambda -\xi_{j}|^{1/2}) \quad \textrm{as }  \lambda \to \xi_{j}
\end{equation}
for every $\xi_j\in\cup_{i=1}^{m} \varphi^{-1}(q_i)$ and a local coordinate at $q_i$ $(i=1, 2, \dots , m)$ (cf.\ \cite{Pommerenke1991}).

\medskip

\textbf{Step 2: Radial structures.}

	Let $\phi : \overline{X}\times (E\cup\overline{\Delta}_{r}\cup\Delta_{R}^{c})\to \widehat{\mathbb{C}}$ be a holomorphic motion of $E\cup\overline{\Delta}_{r}\cup\Delta_{R}^{c}$ over $\overline{X}$ as in Theorem \ref{them:reduction}.
	We consider a new holomorphic motion $\phi ':\overline X\times (E\cup\{|z|\leq 2r\}\cup\{|z|\geq R-1\})\to \widehat{\mathbb{C}}$ by
	\begin{equation*}
		\phi ' (p, z)=
		\begin{cases}
			\phi (p, z), & z\in E\\
			z, & z\in \{|z|\leq 2r\}\cup\{|z|\geq R-1\}.
		\end{cases}
	\end{equation*}

As the monodromy of $\phi$ is trivial, the monodromy of $\phi '$ is also trivial; it follows from Proposition \ref{Pro:tameQC} that there exists a real analytic map $\Phi : \overline{X}\to M(\mathbb{C})$ such that ${\phi}_0 : \overline{X}\times \widehat{\mathbb{C}}\to \widehat{\mathbb{C}}$ defined by
\begin{equation}
  {\phi}_{0}(p, z)=w^{\Phi(p)}(z)
\end{equation}
extends $\phi '$. 
From Proposition \ref{Pro:tameQC} (2), $\phi_0 (p, \cdot)$ is real analytic on $\{2r<|z|<R-1\}\setminus E$. 
Since $E$ is a finite set, $\phi_0 (p, \cdot)$ is real analytic on $\{2r<|z|<R-1\}$.
We may modify the map to have a smooth map on $\{r\leq |z|\leq R\}$ which extends $\phi_0$; we give a concrete construction of the extension.

For each $p\in \overline X$, we put $S_{p}=\{2r<|z|<R-1\}\setminus \{\phi (p. z_1), \dots , \phi (p, z_n)\}$.
The domain $S_p$ is a Riemann surface of type $(0; n, 2)$ and a map
$f_p:=\phi_{0}(p, \cdot)|S_{p_0}$ is a quasiconformal map from $S_{p_0}$ to $S_p$. 
From Proposition \ref{Pro:symmetricQC}, there exists a real analytic symmetric quasiconformal map $e(f_p) : S_{p_0}\cup\partial S_{p_0}\to S_p\cup\partial S_{p}$ which determines the same point in the reduced Teichm\"uller space $Teich^{\#}(S_{p_0})$ as $f_p$.
In particular, $e(f_p)(z)=\phi (p, z)$ for any $w\in E\cap S_{p_0}$.
Furthermore, it follows from the symmetricity that $e(f_p)(e^{i\theta}[2r, (R-1)])$ is orthogonal to $\{|z|=2r\}$ and $\{|z|=R-1\}$ for any $e^{i\theta}\in \partial\Delta$, where $e^{i\theta}[2r, (R-1)]$ is the closed line segment between $2re^{i\theta}$ and $(R-1)e^{i\theta}$. 
However, $e(f_p)|\partial S_{p_0}$ may not be the identity.
We extend the map to $\{R-1\leq |z|\}$ by the following way.

Put $e(f_p)((R-1)e^{i\theta}))=(R-1)e^{i\Theta_p(\theta)}$ $(\Theta_p(\theta)\in [0, 2\pi))$ for $\theta \in [0, 2\pi)$.
$\Theta_p$ is a homeomorphism from $[0, 2\pi)$ onto $[\Theta_p (0), \Theta_p (0)+2\pi)$.
We define a quadratic polynomial $g_p^{\theta}$ for each $\theta\in [0, 2\pi)$ by
$$
g_p^{\theta}(x)=A_p (x-\log (R-1))^2+\Theta_p(\theta),
$$
where $A_p=(\theta-\Theta_p (\theta))(\log R-\log (R-1))^{-2}$; we see $g_p^{\theta} (\log (R-1))=\Theta_p(\theta)$, $(g_{p}^{\theta})'(\log (R-1))=0$ and $g_p^{\theta} (\log R)=\theta$.
Then, we define a map $\phi_1 (p, w)$ for $(p, z)\in \overline{X}\times \{2r\leq |z|\}$ by
\begin{equation*}
	\phi_1 (p, z)=
	\begin{cases}
		z, & (|z|>R) \\
		e(f_p)(z), & (2r\leq |z|<R-1) \\
		|z|\exp \{ig_{p}^{\theta}(\log |z|)\}, & (R-1\leq |z|\leq R, \arg z=\theta\in [0, 2\pi)).
	\end{cases}
\end{equation*}
It is not hard to see that $g_{p}^{\theta}(x)\not= g_{p}^{\theta '}(x)$ if $\theta\not=\theta '$.
Hence, the map $\phi_1$ is a diffeomorphism from $\overline X\times \{2r\leq |z|\leq R\}$ onto itself. 
In particular, $\phi_1 (p, \cdot)$ is quasiconformal and $\phi_1 (p, e^{i\theta}((R-1), R))$ is orthogonal to $\{|z|=R-1\}$ at $w=e(f_p)((R-1)e^{i\theta})$.

Furthermore, there exist $\epsilon >0$ and $c >0$ such that for any $(p, \theta)\in \overline X\times [0, 2\pi)$, the image of $\{ z=\rho e^{i\theta}\mid R-c <\rho \leq R\}$ via $\phi_1(p, \cdot)$ is in the Stolz region at $Re^{i\theta}$ for $\epsilon$.
Namely, 
\begin{equation*}
	\frac{R-\rho}{|Re^{i\theta}-\phi_1 (p, \rho e^{i\theta})|}>\epsilon
\end{equation*}
holds if $R-c < \rho\leq R$.

We may extend $\phi$ to $\overline X\times \{r\leq |z|\leq 2r\}$ by a similar way. 
We denote the extended map also by $\phi_1$. 
Then, at $re^{i\theta}$, $\phi_1 (\{ z=\rho e^{i\theta}\mid r\leq \rho <r+c\} )$ is in the "Stolz region" for some $\epsilon>0$;
\begin{equation*}
	\frac{\rho - r}{|re^{i\theta}-\phi_1 (p, \rho e^{i\theta})|}>\epsilon
\end{equation*}
holds if $r \leq \rho < r+c$.

\medskip

Thus, we have the map $\phi_1$ having the following properties:
\begin{enumerate}
\renewcommand{\labelenumi}{(\alph{enumi})}
  \item $\phi_{1} : \overline{X}\times \widehat{\mathbb{C}}\to \widehat{\mathbb{C}}$ is a quasiconformal motion over $\overline{X}$, that is, a continuous family of quasiconformal maps parametrized by $\overline{X}$; 
  \item for each $p \in \overline X$ and $\zeta\in\partial\Delta_{R}$, $\phi_{1}(p, \zeta [r/R, 1])$ is a differentiable arc, where $\zeta [a, b]$ $(0<a<b)$ is the closed line segment from $\zeta a$ to $\zeta b$;
  \item $\phi_{1}(p, z)=\phi (p, z)$ for $(p, z)\in \overline{X}\times E$;
  \item $\phi_1(p, z)=z$ if $z\in \Delta_R\setminus\overline{\Delta_r}$. 		
\end{enumerate}

Naturally, the map $\phi_1$ is considered in $\overline {X'}\times \widehat{\mathbb C}$.
The set of simple arcs $\ell({\xi}, {\zeta}):={\phi}_{1}(\xi, \zeta I)$ for $I =[0, 1]\subset \mathbb{R}$ $(\xi\in\partial X', \zeta\in\partial \Delta_{R})$ is called the \lq\lq radial structure\rq\rq in \cite{Slodkowski1995}.
\begin{lemma}
\label{lemma:RadialStructure}
	$\{\ell (\xi, \zeta)\mid \xi\in \partial X', \zeta\in \partial\Delta_{R} \}$ has the following properties.
\begin{enumerate}
  \item $\ell (\xi, \zeta)$ is a differentiable arc connecting $0$ and $\zeta$;
  \item $\cup_{|\zeta|=R}\ell (\xi, \zeta)=\overline{\Delta}_{R}$ for each $\xi\in\partial X'$;
  \item $\ell (\xi, \zeta)\cap\ell (\xi, \zeta')=\{0\}$ if $\zeta\not=\zeta'$;
  \item there exist $\zeta_{1}, \dots , \zeta_{n}$ on $\partial \Delta_{R}$ such that $\phi (\xi, z_{i})\in  \ell (\xi, \zeta_{i})$ $(i=1, \dots , n)$ for every $\xi\in \partial X'$;
  \item for any $\xi\in \partial X'$, $\ell (\xi, \zeta)\cap\overline{\Delta}_{r}=\zeta [0, r]$, the line segment from $0$ to $r\zeta$;
  \item Let $\xi ', \xi ''$ be two distinct points on $\partial X'$ corresponding to the same point $\xi$ in $X$. Then, $\ell (\xi ', \zeta)=\ell (\xi'', \zeta)$ for any $\zeta\in\partial \Delta_{R}$;
  \item there exist $\epsilon>0$ and $c>0$ such that the images of $\ell (\xi, \zeta) \cap\{|\zeta-z|<c\}$ and $\ell (\xi, \zeta) \cap\{|r\zeta/R-z|<c \}$ via $\phi_1(\xi, \cdot)$ are in the Stoltz regions at $\zeta$ and $r\zeta/R$ for $\epsilon$, respectively.
\end{enumerate}

\end{lemma}
\begin{proof}
Obvious are (2), (3) and (7) from the construction. 
The statement (5) is also easily verified because $\phi_{1}(\xi, z)=z$ if $z\in \Delta_{r}$.

Taking $\zeta_{i}=z_{i}R/|z_{i}|$ $(i=1, \dots , n)$, we see that $\phi (\xi, z_{i})={\phi}_1(\xi, z_{i})={\phi}_1(\xi, \zeta_{i}(|z_{i}|/R))\in \ell (\xi, \zeta_{i})$. 
Hence we have (4). Now, we show (1).

Since $\Phi : \overline{X} \to M(\mathbb{C})$ is real analytic, $\overline{X}\ni x\mapsto w^{\Phi (x)}(z)\in \widehat{\mathbb{C}}$ is real analytic on $\overline{X}$ for each $z\in \mathbb{C}$.
We also see from Proposition \ref{Pro:tameQC} (2) that $\Phi (x)\in M(\mathbb{C})\subset L^{\infty}(\mathbb{C})$ is real analytic in $\mathbb{C}\setminus (E\cup\partial \Delta_{R})$.
Therefore, $w^{\Phi (x)}(z)$ is real analytic with respect to $z$ in $\mathbb{C}\setminus (E\cup\partial \Delta_{R})$ (see \cite{Nag1988}).
Furthermore, $w^{\Phi(x)}$ is real analytic at every point of $E$ since $E$ is a finite set; we conclude that $w^{\Phi(x)}$ is real analytic except $\partial \Delta_{R}$.
Thus, we have shown (1).

The statement (6) is trivial because $\phi_{1} (\xi', \cdot)=\phi_1 (\xi, \cdot)= \phi_{1}(\xi '', \cdot )$.

\end{proof}

\textbf{Step 3: A function space.}

For a compact bordered Riemann surface $S$, let denote by $A_{S}$ the space of holomorphic functions on a some neighborhood $U$ of $\overline{S}$. 
We define $A(X')$ by $A_{\Delta}\circ\varphi^{-1}$, where $\varphi$ is a Riemann map given in Step 1. Note that $X'$ is simply connected and $X'\cup\partial X'$ is homeomorphic to $\overline{\Delta}$ via the Riemann map $\varphi :\Delta\to X'$.

For each $\zeta\in \partial\Delta_{R}$, we define
\begin{equation}
\label{dfn:F}
  \mathcal{F}^{\zeta}=\{f\in A(X ')\mid f(\xi)\in \ell (\xi, \zeta) \textrm{ for every } \xi\in \partial X' \},
\end{equation}
and
\begin{equation}
  \mathcal{F}=\cup_{\zeta\in \partial\Delta_{R}}\mathcal{F}^{\zeta}.
\end{equation}
Functions $f\equiv 0$, $\equiv \zeta$ and $\equiv r\zeta/|\zeta|$ are contained in $\mathcal{F}^{\zeta}$ for every $\zeta\in \partial \Delta_{R}$; it follows from Lemma \ref{lemma:RadialStructure} (4) that $\phi (\varphi (\cdot) , \zeta_{i}) \in \mathcal{F}^{\zeta_{i}}$ $(i=1, \dots, n)$.

We denote by $C^{\frac{1}{2}}(\partial \Delta)$ the space of complex valued H\"older continuous functions of exponent $\frac{1}{2}$, and by $C^{\frac{1}{2}}_{\mathbb{R}}(\partial \Delta)$ its subspace of all real valued functions.

One may find similar statement to the following lemma in  \cite{Slodkowski1995} Lemma 1.~3; we give a different and an elementary proof. 

\begin{lemma}
	\label{lemma:Fa-class}
$\mathcal{F}^{\zeta}\circ\varphi$ and $\mathcal{F}\circ\varphi$ are regarded as compact subsets of $C^{\frac{1}{2}}(\partial \Delta)$ by considering boundary values.
Moreover, the followings hold:
	\begin{enumerate}
  \item If $g\in \mathcal{F}^{\zeta}$ and $\|g\|_{\infty}=R$, then $g(p)\equiv \zeta$; 
  \item if $g\in \mathcal{F}^{\zeta}$ and $|g(p)|\leq r$ for some $p\in \overline{X}$, then $g$ is a constant function.
   \end{enumerate}

\end{lemma}
\begin{proof}	
	Since $f\in A(X')$ is continuous on $X'\cup \partial X'$ and $\varphi : \Delta\cup\partial \Delta\to X\cup\partial X'$ is a H\"oler continuous function of exponent $\frac{1}{2}$ by (\ref{eqn:Holder}), we see that $f\circ\varphi\in A_{\Delta}$ belongs to $C^{\frac{1}{2}}(\partial \Delta)$. 
	Thus, $\mathcal{F}^{\zeta}\circ\varphi$ and $\mathcal{F}\circ\varphi$ are in $C^{\frac{1}{2}}(\partial \Delta)$.
	The compactness of those spaces is easily verified.

	Let $g\in\mathcal{F}^{\zeta}$ be a non-constant function with $\|g\|_{\infty}=R$. 
	From (\ref{dfn:F}) we see that $g\circ\varphi (\partial\Delta )\subset \cup_{\xi\in\partial X'}\ell (\xi, \zeta)$ and $g\circ\varphi (\partial\Delta )\cap\partial\Delta_{R}=\{\zeta\}$.	
	Hence, there exists $\lambda \in\partial \Delta$ such that $g(\varphi (\lambda ))=\zeta$.

	Now, we recall that $\ell (\xi, \zeta)$ is the image of the line segment $\zeta [0, 1]$ via a quasiconformal map $\phi_1 (\xi, \cdot)$ which keeps every point of $\partial\Delta_{R}$ fixed.
	From our construction of the radial structures, we see that the angle between $\ell (\xi, \zeta)$ and $\partial\Delta_{R}$ at $\zeta$ is greater than a positive constant $\epsilon$ which is independent of $\xi$ (Lemma \ref{lemma:RadialStructure} (7)). 
	Therefore, the angle between $g(\partial X)\subset \cup_{\xi\in\partial X'}\ell (\xi, \zeta)$ and $\partial\Delta_{R}$ at $\zeta$ is greater than $\epsilon>0$, and the angle of $g\circ\varphi (\partial\Delta )$ at $\zeta$ is less than $\pi-2\epsilon$. 
	This means that the order of $g\circ\varphi$ at $\lambda$ is greater than one since $g\circ\varphi$ is holomorphic on $\partial \Delta$. 
	Indeed, if the order is one, then $(g\circ \varphi)'(\lambda )\not=0$ and $g\circ\varphi$ is conformal at $\lambda$; the angle of $g\circ\varphi (\partial\Delta)$ at $
	\lambda$ has to be $\pi$.
	
	Thus, there exists a point $p'\in X$ near $p=\varphi^{-1}(\lambda)$ such that $|g(p')|>R$ and we have a contradiction.
	We complete the proof of (1).
	
	To show (2), we first assume that there exists a non-constant function $g\in \mathcal{F}^{\zeta}$ such that it has a zero in $X'$ but not on $\partial X'$.
	
	Since the monodromy of $\phi_{1}$ is trivial, we see that for every closed curve $\gamma$ in $\overline{X}$ and for every $x\in (0, \zeta)$, $\phi_{1}(\gamma, x)$ is freely homotopic to a trivial curve in $\widehat{\mathbb{C}}\setminus\{0, 1, \infty\}$. 
	In particular, the winding number of the curve $\phi_{1}(\partial X', x)\subset \cup_{\xi\in \partial X'}\ell (\xi, \zeta)$ around the origin is zero.
	
	Let us consider the winding number of a closed curve $g(\partial X')$ around the origin.
	Since the function $g$ belongs to $\mathcal{F}^{\zeta}$, both $g(\xi)$ and $\phi_{1}(\xi, x)$ are on the simple arc $\ell (\xi, \zeta)$ for every $\xi\in \partial X'$.
	By moving $g(\xi)$ to $\phi_{1}(\xi, x)$ along $\ell (\xi, \zeta)$, we have a homotopy between $g(\partial X')$ and $\phi_{1}(\partial X', x)$.
	Since $\phi_1 (\partial X' , x)$ does not pass the origin, we conclude that the winding number of $g(\partial X')$ around the origin is also zero.
	Hence, the holomorphic function $g$ has a zero in $X'$ by the argument principle and we have a contradiction.
	
	Therefore, if $g\circ\varphi$ has zeros in $\overline{\Delta}$, then it has a zero on $\partial \Delta$.
	Noting that 
	$\ell (\xi, \zeta)\cap\{|z|<r\}=\zeta[0, r/R)$ for every $\xi\in \partial X'$ and $\zeta\in \partial\Delta_R$, we see that
	$$
	g\circ \varphi (\partial\Delta)\cap\{|z|<r\}=\zeta[0, r/R].
	$$
	Hence, the order of any zero of $g\circ \varphi$ on $\partial\Delta$ is even.

	We note the following lemma:

	\begin{lemma}
	\label{lemma:realpositive}
		Let $f$ be a holomorphic function in a neighborhood of an analytic closed curve $\gamma$. Suppose that the order of any zeros of $f$ on $\gamma$ is even and there exist a point $z_0 \in \gamma$ and a neighborhood $U$ of $z_0$ such that $f(\gamma\cap U)\subset \mathbb{R}_{\geq 0}$. Then, $f(\gamma)\subset \mathbb{R}_{\geq 0}$. 
	\end{lemma}
	\begin{proof}
	We may assume that $\gamma=\mathbb{R}$ and $z_{0}=0$. Let $I$ be the maximal interval on $\mathbb{R}$ containing $0$ such that $f(z)\geq 0$ for every $z\in I$. 
	The interval $I$ is obviously a closed interval. 
	Suppose that $I=[a, b]$ $(a<0<b)$. If $f(a)=0$, then we have
	\begin{equation*}
		f(z)=c(z-a)^{2n}(1+\sum_{k=1}^{\infty}c_{k}(z-a)^{k})
	\end{equation*}
	near $z=a$ for some $n\in \mathbb{N}$ and a constant $c$. 
	Since $f(x)>0$ for $x>a$ sufficiently close to $a$, we have $c>0$.
	Thus, we see that $f(x)>0$ for $x<a$ sufficiently close to $a$. It contradicts the maximality of $I$.
	
	If $c_0:=f(a)>0$, then we have
	\begin{equation*}
		f(z)=c_0 + \sum_{k=1}^{\infty} c_{k}(z-a)^{k}
	\end{equation*}
	near $z=a$, say $(a-\epsilon, a+\epsilon)$ for some $\epsilon >0$.  
	Since $f(x)>0$ for $a <x < a+\epsilon$, we see that $c_1 = f'(a)\in \mathbb{R}$.
	We also verify that $f'(x)$ is real in $[a, a+\epsilon)$ because $f(x)$ is real there. 
	Thus, $c_2=f''(a)/2$ is real. 
	The same argument yields that $c_{k}$ is a real number for any $k\in \mathbb{N}$. 
	Hence, $f$ is real valued in a neighborhood of $a$ on $\mathbb R$.
	From the continuity of $f$, we conclude that $f(x)>0$ near $x=a$. This contradicts the maximality of $I$.
	
	The same argument works for $b$ and we conclude that $I=\gamma$.
	\end{proof}
	
	Applying Lemma \ref{lemma:realpositive} for $\gamma =\partial\Delta$ and $\zeta^{-1}(g\circ\varphi)$, we see that $\textrm{Im }\zeta^{-1}(g\circ\varphi)\equiv 0$ on $\partial\Delta$ and
	 we have a contradiction. Thus $g\equiv 0$.

Next, we suppose that $0<\min \{|g(p)|\mid p\in \overline{X'}\}<r$. 
Let $q_0$ be a point in $\overline{X'}$ with $|g(q_0)|=\min \{|g(p)|\mid p\in \overline{X'}\}$. 
Applying the maximum principle to $g^{-1}$, we verify that $q_0$ is on $\partial X'$. 
By the same argument as above, we see that the order of $g$ at $q_0$ is even. 
Hence, if $g$ is not a constant, then there exists a point $p'\in X'$ near $q_0$ such that $|g(p')|<|g(q_0)|$ and we have a contradiction.

Finally, we suppose that $\min \{|g(p)|\mid p\in \overline{X'}\}=r$. 
Let $q_0\in \overline{X'}$ be a point with $|g(q_0)|=r$.
By using the maximum principle again, we see that $q_0$ is on $\partial X'$.
Since $g\in \mathcal{F}^{\zeta}$, we have $g(q_0)=r\zeta/|\zeta|$.
We may use the same argument as in the proof of (1) for $g$ and for $\partial \Delta_{r}$.
Then, we conclude that the angle between $g(\partial X')$ and $\partial \Delta_{r}$ is strictly positive at $r\zeta/|\zeta|$ and the order of $g$ at $q_0$ is more that one. 
Hence, we see that there exists a point $p'\in X'$ near $q_0$ such that $|g(p')|<r=|g(q_0)|$.
It is a contradiction and we complete the proof of (2).
\end{proof}

\textbf{Step 4: Differential equations.}

Since $X'$ and $\partial X'$ are identified with $\Delta$ and $\partial\Delta$, respectively, we may discuss our argument in $\Delta\cup\partial\Delta$ instead of $X'\cup\partial X'$. 
In this step, we will make our discussion in $\Delta\cup\partial\Delta$ for the sake of simplicity.

For any $(\xi, z)\in \partial\Delta\times\overline{\Delta_{R}}(\simeq \partial X'\times \overline{\Delta_{R}})$, there exists a unique $\zeta\in \partial \Delta_{R}$ such that $\ell (\xi, \zeta)\ni z$.
Since $\ell (\xi, \zeta)\cap (\overline{\Delta_{R}}\setminus\Delta_{r})$ is differentiable, we may consider the unit tangent vector of $\ell (\xi, \zeta )$ at $z\in \ell (\xi, \zeta)$ with $r\leq |z|\leq R$, where we parametrize the curve $\ell (\xi, \zeta)$ by the length parameter from $\zeta\in \partial\Delta_R$.

We denote by $\tau (\xi, z)$ the unit tangent vector of $\ell (\xi, \zeta)$ at $z\in \ell(\xi, \zeta)$.
Hence, the parametrization $z(t)$ of $\ell (\xi, \zeta)$ satisfies a differential equation:
\begin{equation}
\label{eqn:ODEforEll}
	\frac{d}{dt}z(t)=\tau (\xi, z(t)).
\end{equation}

We see that there exists a differentiable map $\alpha : \partial \Delta \times (\overline{\Delta_{R}}\setminus \Delta_{r})\to \mathbb{R}$ such that
\begin{equation}
	\label{eqn:DefOfAlpha}
	\tau (\xi, z)=\frac{z}{|z|}e^{i\alpha(\xi, z)}.
\end{equation}

Indeed, for every $\zeta \in \partial \Delta_{R}$, $\textrm{arg }\tau (\xi, \zeta)\zeta^{-1}$ is in $(-\pi/2. \pi/2)$. 
Thus, a map 
\begin{equation*}
	\partial\Delta \times \partial \Delta_{R}\ni (\xi, \zeta)\mapsto \zeta^{-1}\tau (\xi, \zeta)/R\in \partial \Delta
\end{equation*}
is homotopic to a constant map.
A continuity argument guarantees us the existence of the function $\alpha$. In fact, $\alpha$ is unique up to an additive constant $2n\pi$ $(n\in \mathbb{Z})$.

We may take $\delta >0$ so small that $\tau$ is extended to a differentiable map on $\partial\Delta \times (\Delta_{R+\delta}\setminus\Delta_{r})$.
It is possible by our concrete construction of $\phi_1$ in Step 2. 

We put
\begin{equation}
	\mathcal{U}:=\{g\in C^{\frac{1}{2}}(\partial\Delta )\cap A_{\Delta}\mid r<|g(\lambda)|<R+\delta \textrm{ for all $\lambda\in \overline{\Delta}$}\}.
\end{equation}
$\mathcal{U}$ is an open subset of $C^{\frac{1}{2}}(\partial\Delta )\cap A_{\Delta}$.

Let $\beta (g)(\xi)=\alpha (\xi, g(\xi))$ for $\xi\in \partial \Delta$ and for $g\in \mathcal{U}$. Then we define a map
\begin{equation}
	\label{Def:MapF}
	F(g)=g\exp \{-T[\beta (g)]+iP[\beta (g)]\}
\end{equation}
for $g\in \mathcal{U}$, where $P[f]$ is the Poisson integral of a real-valued continuous function $f$ on $\partial\Delta$ and $T[f]$ is the Hilbert transform of $f$, that is, $T[f](z)$ is the conjugate harmonic function of $P[f]$ with $T[f](0)=0$.
Then, we show the following:
\begin{lemma}
	\label{lemma:Lip}
	$F(\mathcal{U})\subset C^{\frac{1}{2}}(\partial\Delta )\cap A_{\Delta}$ and $F$ is locally Lipschitz on $\mathcal{U}$.
\end{lemma}
\begin{proof}
	Let $g\in \mathcal{U}$. Then, $F(g)$ gives a holomorphic function on $\Delta$. 
	It follows from a property of the Hilbert transform that $F(g)$ belongs to $C^{\frac{1}{2}}(\partial\Delta)$ (cf. \cite{Duren1970} Theorem 5.8). 
	Hence, $F(\mathcal{U})\subset C^{\frac{1}{2}}(\partial\Delta)\cap A_{\Delta}$. 
	
	Since $\alpha$ is differentiable, so is $\beta$. Thus, a compactness argument shows that $F$ is locally Lipschitz on $\mathcal{U}$.
\end{proof}
Here, we consider a differential equation:
\begin{equation}
	\label{eqn:Differential eqn}
	\frac{dg_{t}}{dt}=F(g_{t}), \quad g_{t_0}=g_0\in \mathcal{U}.
\end{equation}
By the standard fact on ordinary differential equations(cf.~\cite{Dieudonn1969} Chapter X), we see that for any initial value $g_{t_0}\in \mathcal{U}$ the differential equation (\ref{eqn:Differential eqn}) has a unique solution $g_t\in \mathcal{U}$ in some interval $(a_0, b_0)$ with $a_0<t_0<b_0$.

If $g_0\in \mathcal{F}^{\zeta}$ for some $\zeta\in\partial\Delta_{R}$, then $g_t\in \mathcal{F}^{\zeta}$ as long as $r< \min_{\lambda \in\overline{\Delta}}|g_t (\lambda )|\leq \|g_t\|_{\infty}\leq R$.
Indeed, from (\ref{Def:MapF}) and (\ref{eqn:Differential eqn}) we have for each $\xi\in\partial\Delta$
\begin{eqnarray*}
	\frac{d}{dt}g_{t}(\xi)&=& F(g_t (\xi))=g_t (\xi)\exp\{-T(\beta (g_t (\xi)))+iP(\beta (g_t (\xi)))\}\\
	&=& g_t (\xi)\exp\{-T(\alpha (\xi, g_t (\xi))+iP(\alpha (\xi, g_t (\xi)))\}\\
	&=& \frac{g_t (\xi)}{|g_t (\xi)|}|F(g_t (\xi))|\exp \{iP(\alpha (\xi, g_t (\xi)))\}.
\end{eqnarray*}
Noting that $P(f)(\xi)=f(\xi)$ for a continuous function $f$ on $\partial\Delta$, we obtain
\begin{equation*}
	\label{eqn:ODEForline}
	\frac{d}{dt}g_t (\xi)=|F(g_t (\xi))|\tau (\xi, g_t (\xi)).
\end{equation*}
Therefore, for each $\xi\in\partial\Delta$, the function $u_{\xi}(t):=g_t (\xi)$ is a solution of 
\begin{equation}
\label{eqn:ODEofg_t}
	\frac{d}{dt}u=|F(u)|\tau (\xi, u).
\end{equation}
Comparing (\ref{eqn:ODEforEll}) and (\ref{eqn:ODEofg_t}), we verify that $g_{t}(\xi)\in \ell (\xi, \zeta)$ if $g_{0}(\xi)\in \ell (\xi, \zeta)$ for $\xi\in \partial\Delta$.
If $|g_{t}(\lambda )|\leq r$, then it follows from Lemma \ref{lemma:Fa-class} (2) that $g_t$ is a constant function which does not belong to $\mathcal{U}$.

Let $\mathcal{I}$ be the largest interval where $g_{t}$ exists and belongs to $\mathcal{F}^{\zeta}\cap\mathcal{U}$. Then, $\mathcal{I}=(a, b]$ for some $a, b\in \mathbb{R}$. From the maximality of $\mathcal{I}$, we have
\begin{enumerate}
  \item $g_{b}(\xi)=\zeta$ for every $\xi\in \partial\Delta$;
  \item $\lim_{t\to \alpha+}g_{t}(\xi)=r\zeta/|\zeta|$ for every $\xi\in \partial\Delta$.
\end{enumerate}
Thus, we verify that
\begin{equation}
	\label{stmt:homeo}
	\textrm{the map }\mathcal{I}\ni t\mapsto g_{t}(\xi)\in \ell (\xi, \zeta)\cap (\overline{\Delta_{R}}\setminus\overline{\Delta_{r}})
	 \textrm{ is a homeomorphism.}
\end{equation}

\medskip

\textbf{Step 5: Extension.}

Let $(\lambda, z)$ be in $\overline{\Delta}\times (\overline{\Delta_{R}}\setminus\overline{\Delta_{r}})$. 
Since $\cup_{\zeta\in \partial\Delta_R}\ell (\xi , \zeta)=\overline{\Delta_{R}}$
$(\xi\in \partial\Delta)$ and $\ell (\xi, \zeta)\cap\ell (\xi , \zeta')=\{0\}$ if $\zeta\not=\zeta'$, there exists a unique $\zeta\in \partial\Delta_R$ such that $z\in \ell (1, \zeta)$. 

We consider the initial value problem of (\ref{eqn:Differential eqn}) for $g_0\equiv \zeta\in \mathcal{F}^{\zeta}\cap\mathcal{U}$.
From the result in Step 4, $g_{t}$ belongs to $\mathcal{F}^{\zeta}$ for $t\in \mathcal{I}$ and there exists a unique $t({z})\in \mathcal{I}$ such that $g_{t({z})}(1)=z\in \ell (1, \zeta)\cap (\overline{\Delta_{R}}\setminus\overline{\Delta_{r}})$.

Thus, we have $g_{t(z)}\in \mathcal{F}^{\zeta}$ with $g_{t(z)}(1)=z$. 
If $g^*$ be a function of $\mathcal{F}^{\zeta}$ with $g^{*}(1)=z$, then we consider the initial value problem of (\ref{eqn:Differential eqn}) for $g_{t_0}=g^*$.
Take the largest interval $\mathcal{I}^{*}=(\alpha^{*}, \beta^{*}]$ for the problem as in Step 4.
Then, $g^{*}_{\beta^{*}}\equiv \zeta$. From the uniqueness of the initial value problem, we verify that $g^{*}=g_{t(z)}$.

Therefore, we have the following.
\begin{lemma}
\label{lemma:unique}
	For each $z\in \overline{\Delta _{R}}\setminus\overline{\Delta_{r}}$, there exists a unique $\zeta\in\partial\Delta_{R}$ and a unique $g:=g_{t(z)}\in\mathcal{F}^{\zeta}$ such that $g(1)=z$.
\end{lemma}
Let $z, z'$ be two distinct points in $\overline{\Delta_{R}}\setminus \overline{\Delta_{r}}$. 
Then, there exist $\zeta, \zeta'\in \partial\Delta$ such that $z\in \ell (1, \zeta)$ and $z' \in \ell (1, \zeta')$.
Hence, $g_{t(z)}\in \mathcal{F}^{\zeta}$ and $g_{t(z')}\in \mathcal{F}^{\zeta'}$.
We show that $g_{t(z)}(\xi)\not=g_{t(z')}(\xi)$ for any $\xi\in \partial\Delta$.

From the definition of $\mathcal{F}$, $g_{t(z)}(\xi)\in \ell (\xi, \zeta)$ and $g_{t(z')}(\xi)\in \ell (\xi, \zeta')$.
If $\zeta\not=\zeta'$, then $g_{t(z)}(\xi)\not=g_{t(z')}(\xi)$ for any $\xi\in \partial\Delta$ because $\ell (\xi, \zeta)\cap\ell (\xi, \zeta')=\{0\}$ and $g_{t(z)}(\xi), g_{t(z')}(\xi) \not=0$.

If $\zeta=\zeta '$, then both $g_{t(z)}(\xi)$ and $g_{t(z')}(\xi)$ are on the same curve $\ell (\xi, \zeta)$.
However, $t(z)\not=t(z')$. Thus, it follows from (\ref{stmt:homeo}) that $g_{t(z)}(\xi)\not=g_{t(z')}(\xi)$ for $\xi\in\partial\Delta$.

We put $G_{z, z'}:=g_{t(z)}-g_{t(z')}$. The above argument shows that $G_{z, z'}$ is a non-vanishing continuous function on $\partial \Delta$ for each $(z, z')\in (\overline{\Delta_{R}}\setminus\overline{\Delta_{r}})^{2}-\{\textrm{diagonals}\}$ and it is continuous with respect to $(z, z')$.
Hence, the winding number of $G_{z, z'}(\partial\Delta)$ around the origin is independent of $(z, z')$. 
From Lemma \ref{lemma:Fa-class} (1), we have $g_{t(\zeta)}\equiv \zeta\in \partial\Delta_R$. 
Thus, we see that the winding number of $G_{\zeta , z'}(\partial\Delta)=\zeta-g_{t(z')}(\partial\Delta)$ around the origin has to be zero since $g_{t(z')}(\Delta)\subset \Delta_{R}$ and $g_{t(z')}(\partial\Delta)\not\ni \zeta$ if $z'\in \Delta_{R}$.
Obviously, the winding number of $G_{\zeta , z'}(\partial\Delta)=\zeta-g_{t(z')}(\partial\Delta)$ around the origin is zero when $z'=\zeta' (\not=\zeta)$ is on $\partial\Delta_R$. 
It follows from the argument principle that $g_{t(z)}-g_{t(z')}$ does not have zeros in ${\Delta}$ if $z\not=z'$. 
Therefore, the map 
$$
\overline{\Delta_{R}}\setminus\overline{\Delta_{r}}\ni z \mapsto g_{t(z)}(\lambda)\in \overline{\Delta_{R}}\setminus\overline{\Delta_{r}}
$$
is injective and continuous for each $\lambda \in \overline{\Delta}$.

Now, we define a map $\Psi_{\lambda} : \widehat{\mathbb{C}}\to \widehat{\mathbb{C}}$ for each $\lambda \in\overline{\Delta}$ by
\begin{equation}
	\label{dfn:pre-motion}
	\Psi _{\lambda}(z)=	
	\begin{cases}
		g_{t(z)}(\lambda ), \quad & z\in \overline{\Delta_{R}}\setminus\overline{\Delta_{r}} \\
		z, \quad &\textrm{otherwise}.
	\end{cases}
\end{equation}

As we have noted above, $\Psi_{\lambda}$ is a homeomorphism on $\widehat{\mathbb{C}}$ for each $\lambda \in \overline{\Delta}$; we define
\begin{equation}
	\label{dfn:final}
	\widehat{\phi}(p, z)=\Psi_{\varphi^{-1}(p)}\circ \Psi_{0}^{-1}(z)
\end{equation}
for $p\in \overline{X'}$.
We see that $\widehat{\phi}$ is a holomorphic motion of $\widehat{\mathbb{C}}$ over $\overline{X'}$.

Indeed, for the basepoint $p_0\in X'$, we have
$$
\widehat{\phi}(p_0, z)=\Psi_{0}(\Psi_{0}^{-1}(z))=z
$$
as $\varphi ^{-1}(p_0)=0$. 
It is obvious that $\widehat{\phi}(p, \cdot)$ is a homeomorphism for each $p\in\overline{X'}$.

If $z$ is in $\overline{\Delta_R}\setminus\overline{\Delta_r}$, then we have 
$$
\widehat{\phi}(p, z)=g_{t(z_0)}(\varphi^{-1}(p)),
$$
where $z_0$ is the point in $\overline{\Delta _R}\setminus{\overline{\Delta_r}}$ with $\Psi_{0}(z)=z_0$: $\hat{\phi}(p, z)$ is holomorphic with respect to $p\in \overline{X'}$.

If $z$ is not in $\overline{\Delta_R}\setminus\overline{\Delta_r}$, then $\widehat{\phi}(p, z)=z$.
Thus, we see that $\widehat{\phi}(\cdot, z)$ is holomorphic in $\overline{X'}$ and we verify that $\widehat{\phi}$ is a holomorphic motion of $\widehat{\mathbb{C}}$ over $\overline{X'}$.

Finally, we see that $\widehat{\phi}(p, \cdot)$ agrees with $\phi (p, \cdot)$ on $E=\{0, 1, \infty, z_1, z_2, \dots , z_n\}$ for any $p\in X'$.
For $z_j\in E$ $(j=1, 2, \dots, n)$, there exists $\zeta_{j}\in \partial \Delta_{R}$ such that $\phi (\varphi (\cdot), z_{j})\in \mathcal{F}^{\zeta_{j}}$ (see Lemma \ref{lemma:RadialStructure} (4)). 
In particular, $\omega_j:=\phi (\varphi (1), z_j)$ is on $\ell (\varphi (1), \zeta_j)$ by the definition of $\mathcal{F}^{\zeta_j}$, where $\varphi : \Delta \to X'$ is the Riemann map given in Step 1.
Since both $g_{t(\omega_j)}$ and $\phi (\varphi (\cdot), z_j)$ belong to $\mathcal{F}^{\zeta_j}$ and take the same value $\omega_j$ at $1$, it follows from the uniqueness of Lemma \ref{lemma:unique} that 
\begin{equation*}
	\phi(\varphi (\cdot), z_j)=g_{t(\omega_{j})}(\cdot)
\end{equation*}
on $\Delta$.
Hence, we have
\begin{equation}
	\label{eqn:confirmExtent1}
	\phi(\varphi(\lambda ), z_{j})=g_{t(\omega_{j})}(\lambda )=\Psi_{\lambda} (\omega_{j})=\Psi_{\lambda} (\phi(1, z_{j}))
\end{equation}
and
\begin{equation}
	\label{eqn:confirmExtent2}
	\Psi_0 (\phi(1, z_j))=\phi(p_0, z_j)=z_j.
\end{equation}
Therefore, we conclude from (\ref{eqn:confirmExtent1}) and (\ref{eqn:confirmExtent2}) that
\begin{equation}
	\widehat{\phi}(p, z_j)=\Psi_{\lambda}(\Psi_{0}^{-1}(z_j))=\Psi_{\lambda}(\phi(1, z_j))=\phi (p, z_j), \quad (\lambda =\varphi^{-1}(p))
\end{equation}
for $p\in {X'}$ $(j=1, 2, \dots , n)$.
That is, both holomorphic motions agree at $E$.

\medskip

\textbf{Step 6 : Getting a holomorphic motion over the whole Riemann surface $X$.}

In Step 5, we obtain a holomorphic motion $\widehat{\phi}: \overline{X'}\times\widehat{\mathbb{C}}\to \widehat{\mathbb{C}}$.
Now, we show that $\widehat{\phi}$ becomes a holomorphic motion of $\widehat{\mathbb{C}}$ over $X$. 

Let $\xi^{+}, \xi^{-}$ be two points on $\partial X'$ coming from the same point $\xi$ in $X$ (see {\textsc{Figure 2}}) and $z$ be a point in $\overline{\Delta}_{R}\setminus\overline{\Delta}_{r}$.
There exists a unique $\zeta\in \partial\Delta_{R}$ such that $z$ is on $\ell (1, \zeta)$.
We find a function $g_{t(z)}\in \mathcal{F}^{\zeta}$ appearing in (\ref{dfn:pre-motion}) is the function in flow $\{g_{t}\}_{t\in \mathcal{I}}$ at time $t=t(z)$ such that $g_{t(z)}(1)=z$. 
It is obtained from the solution of the differential equation (\ref{eqn:Differential eqn}) for the initial value $g_0\equiv \zeta$.

We also see that $\mathcal{I}\ni t\mapsto g_{t}(\xi^{\pm})\in\ell (\xi^{\pm}, \zeta)$ gives  homeomorphisms of $\mathcal{I}$ onto $\ell(\xi^{\pm}, \zeta)$ which are simple arcs from $\zeta$ to $r\zeta/|\zeta|$.
On the other hand, from Lemma \ref{lemma:RadialStructure} (6), we have $\ell (\xi^{+}, \zeta)=\ell (\xi^{-}, \zeta)$ for any $\zeta\in \partial \Delta_{R}$.
Hence, both $t\mapsto g_{t}(\xi^{+})$ and $t\mapsto g_{t}(\xi^{-})$ are solutions of the same differential equation with the same initial value $\zeta$. 
Thus, $g_{t}(\xi^{+})=g_{t}(\xi^{-})$ for every $t\in \mathcal{I}$. 
In particular, $g_{t(z)}(\xi^{+})=g_{t(z)}(\xi^{-})$.
We verify that $g_{t(z)}(\xi):=g_{t(z)}(\xi^{\pm})$ is well defined; $\widehat{\phi}(\cdot, z)$ is a continuous function on $X$.

For each $z\in \widehat{\mathbb{C}}$, $X\ni p\mapsto \widehat{\phi}(p, z)$ is continuous on $X$ and holomorphic on $X'=X\setminus\cup_{i=1}^{k}\alpha_{i}$.
Since $\alpha_{i}$ $(i=1, 2, \dots , k)$ are analytic curves, it follows from a fundamental result of complex analysis that $\widehat{\phi} (\cdot, z)$ must be a holomorphic function on $X$. 
Hence, we have obtained a desired holomorphic motion $\widehat{\phi}$ of $\widehat{\mathbb{C}}$ over $X$ which extends $\phi$.
We have completed the proof of the theorem.

\section{Proof of Theorem \ref{thm:mythmTrace}}
We prove both statements (1) and (2) by constructing examples.

Let $E$ be $\{0, 1, \infty, z_0, z_1, \dots , z_n\}$. 
We consider a condition for the monodromy of a holomorphic motion of $E$ to be trivial.

\begin{lemma}
\label{lemma:Kra}
	Let $f$ be a holomorphic function on a Riemann surface $X$ with a basepoint $p_0$ such that
	$f(X)\cap (E\setminus\{z_0\}) =\emptyset$ and $f(p_0)=z_0$. 
	Let $\phi_{f}$ be a holomorphic motion of $E$ over $X$ defined by
	\begin{equation}
		\phi_{f} (p, z)=\begin{cases}
			z, \quad &(z\in E\setminus\{z_0\}), \\
			f(p), \quad &(z=z_0).
		\end{cases}
	\end{equation}
	Then, the monodromy of $\phi_{f}$ is trivial if and only if for any closed curve $\gamma$ in $X$,$f(\gamma)$ is homotopic to the trivial curve in $\mathbb{C}\setminus\{0, 1, z_1, \dots , z_n\}$. 
\end{lemma}
\begin{proof}
	It is obvious that $\phi_{f}$ is a holomorphic motion of $E$ over $X$. 
	
	To consider the monodromy of $\phi_{f}$, we take a closed curve $\gamma$ on $X$ passing through $p_0$ and the lift $\tilde{\gamma}$ of $\gamma$ to the universal covering $\Delta$ of $X$.
	For the universal covering map $\pi : \Delta \to X$, $\tilde{\gamma}$ is a simple arc connecting two points $\zeta_0, \zeta_1\in \pi^{-1}(p_0)$. 
	
	By restricting the holomorphic motion $\tilde \phi_{f}(\cdot, z) := \phi_{f} (\pi (\cdot), z)$ 
	$(z\in E)$ to a simply connected neighborhood $U$ of $\tilde{\gamma}$ in $\mathbb{H}$, we have a holomorphic motion $\tilde \phi_{f}|U$ of $E$ over $U$ with basepoint $\zeta_0$. 
	Since $U$ is a simply connected domain which is conformally equivalent to the unit disk $\Delta$, it follows from Theorem \ref{thm:old} (2), $\tilde \phi_{f}|U$ is extended to a holomorphic motion, say $\hat{\phi}_U$, of $\widehat{\mathbb{C}}$ over $U$. 
	
	We have a continuous family $\{\varphi _{t}\}_{0\leq t\leq 1}$ by $\varphi_{t}(z)=\hat{\phi}_U (\tilde{\gamma} (t), z)$, where $\tilde{\gamma} (t)$ is a parametrization of $\tilde{\gamma}$ with $\tilde{\gamma}(0)=\zeta_0$ and $\tilde{\gamma}(1)=\zeta_1$. 
	Moreover, each $\varphi_{t}$ is a quasiconformal self-map of $\widehat{\mathbb{C}}$ and $\varphi_0=id$.
	The map $\varphi_1$ determines the monodromy for $\gamma$.
	
	Since $\varphi_t$ fixes each point of $E\setminus\{w_0\}$, $\{\varphi_t\}_{0\leq t\leq 1}$ gives a homotopy between $\varphi_0=id$ and $\varphi_1$.
	Therefore, $\varphi_1$ is homotopic to the identity in ${\mathbb{C}}\setminus \{0, 1, z_1\dots , z_n\}$.
	Thus, it follows from \cite{Kra1981} that $\varphi_1$ is homotopic to the identity on $\widehat{\mathbb{C}}\setminus E$ rel $E$ if and only if $\cup_{t\in [0, 1]}\varphi_{t}(z_0)=\phi (\pi (\tilde{\gamma}(t)), z_0)=f(\gamma)$ is homotopic to the trivial curve.	
\end{proof}

{\bf{Proof of (1).}}
We take closed curves $\gamma_0, \gamma_1$ as in {\sc{Figure}} \ref{Figure:Chirka} and put $\gamma=\gamma_{1}\gamma_{0}\gamma_{1}^{-1}\gamma_{0}^{-1}$.
The curve $\gamma$ represents a non-trivial element of $\pi_{1}(\mathbb{C}\setminus\{0, 1, z_1, \dots , z_n\}, z_0)$.
\begin{figure}
\centering
\includegraphics[scale=.45, bb=400 200 500 500]{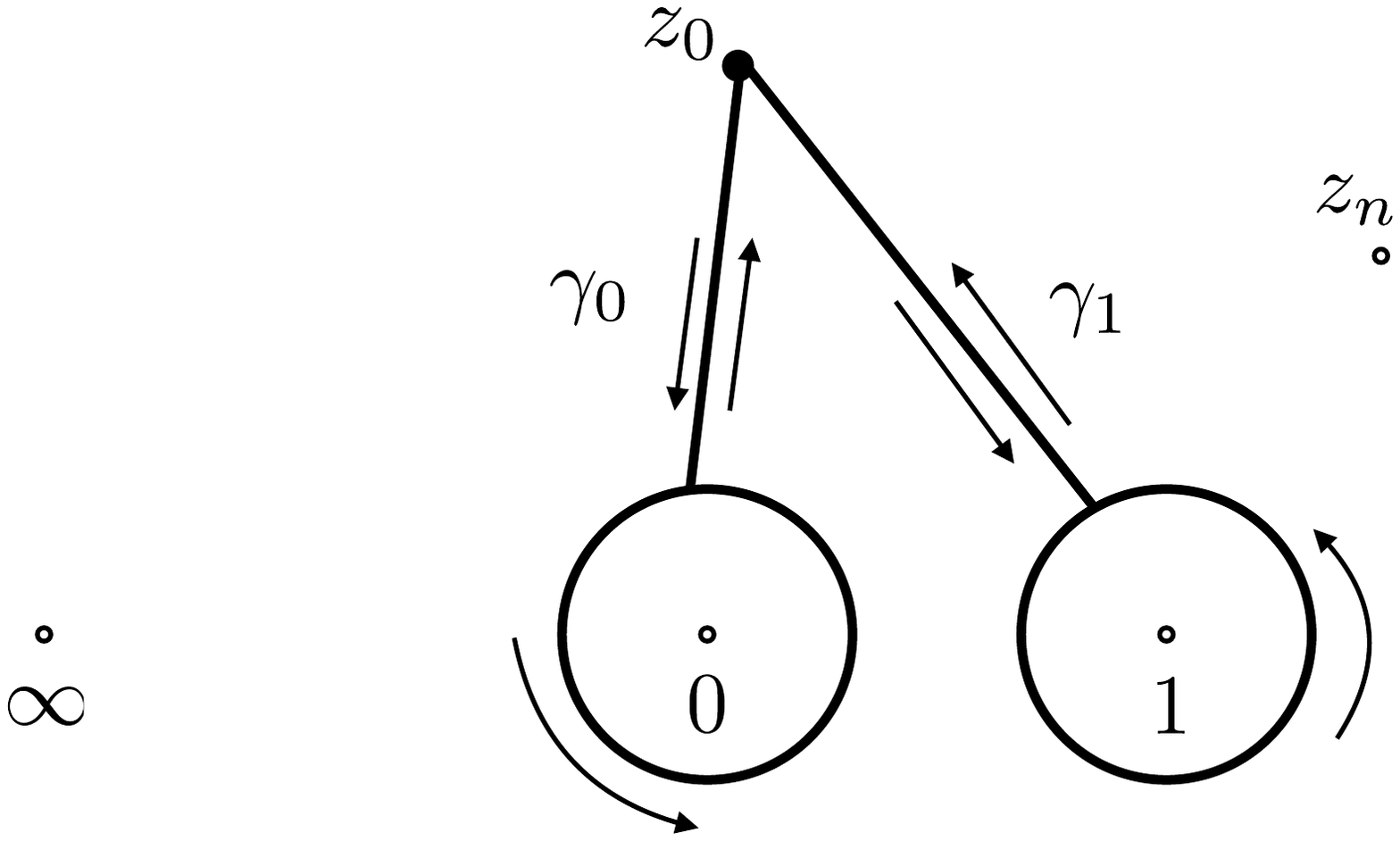}
\caption{}
\label{Figure:Chirka}
\end{figure}

For a parametrization $\gamma (t)$ $(0\le t\le 1)$ of $\gamma$ with $\gamma(0)=\gamma(1)=w_0$, we have
\begin{equation}
		\label{eqn:winding=0}
	\int_{\gamma}d\textrm{arg}(w-1)=\int_{\gamma}d\textrm{arg}w=0.
\end{equation}

Let $\Gamma$ be a Fuchsian group acting on $\Delta$ such that $\Delta/\Gamma=\mathbb{C}\setminus\{0, 1, z_1, \dots , z_n\}$. 
Take $g\in \Gamma$ corresponding to $\gamma$. 
Then, $g$ is a hyperbolic M\"obius transformation and $A:=\Delta/<g>=\{1<|\lambda |<R\}$ is an annulus for some $R>1$; we have a holomorphic covering map $\pi_{g} : A\to \mathbb{C}\setminus\{0, 1, z_1, \dots , z_n\}$. 
We take $\lambda _0\in A$ so that $\pi_{g}(\lambda _0)=z_0$.
	
We take the point $\lambda _0\in A$ as a basepoint and define a map $\phi : A\times E\to \widehat{\mathbb{C}}$ by
\begin{equation}
	\phi (\lambda , z)=
	\begin{cases}
		z, \quad (z=0, 1, \infty, z_1, \dots , z_n) \\
		\pi_{g}(\lambda ), \quad (z=z_0).
	\end{cases}
\end{equation}
It is easy to see that $\phi$ is a holomorphic motion of $E$ over the Riemann surface $A$.

Consider $\alpha=\{|\lambda |=|\lambda _0|\}\subset A$. Then, $\phi (\alpha, z_0)$ is a closed curve in $\mathbb{C}\setminus\{0, 1\}$ homotopic to $\gamma$.
It follows from (\ref{eqn:winding=0}) that
\begin{equation*}
	n(\alpha; z, z')=0
\end{equation*}
for any distinct points $z, z'$ in $E$. 

Since the homotopy class of $\alpha$ generates $\pi_1 (A, \lambda_0)$, the holomorphic motion $\phi$ satisfies Chirka's condition.
However, Lemma \ref{lemma:Kra} implies that
the monodromy of $\phi$ is not trivial.
Hence, the holomorphic motion $\phi$ cannot be extended to a holomorphic motion of $\widehat{\mathbb{C}}$ over $A$ and we obtain a desired example.

\medskip

{\bf{Proof of (2).}}
The proof of (2) is done by using the same idea as in the proof of (1), but it is a bit complicated.

If the set $E$ consists of four points, then the above example constructed in (1) for $n=0$ is a desired one because any holomorphic motion of three points can be extended to a holomorphic motion of $\widehat{\mathbb C}$. 
Therefore, we may assume that $n\geq 1$.

For $E=\{0, 1, \infty, z_0, z_1, \dots , z_n\}$, we take $\gamma_0, \gamma_1, \dots , \gamma_{n+1}$ in {\sc Figure} \ref{FigurePB}. 
The homotopy classes of them are generators of a free group $F_{n+2}:= \pi_{1}(\widehat{\mathbb{C}}\setminus (E\setminus\{z_0\}), z_0)$.

We put
$$
\beta =\gamma_0\gamma_1\dots\gamma_{n+1}.
$$
We define a sequence of closed curves, $\tilde{\gamma}_1, \tilde{\gamma}_2, \dots , \tilde{\gamma}_{n+1}$ by
\begin{equation}
	\label{eqn:tilde-gammas}
	\tilde{\gamma}_1=[\gamma_0, \gamma_1], \quad \mbox{and} \quad \tilde{\gamma}_{j}=[\gamma_{j}, \tilde{\gamma}_{j-1}] \quad (j=2, \dots, n+1)
\end{equation}
where $[a, b]=aba^{-1}b^{-1}$.
\begin{figure}
\centering
\includegraphics[scale=.45, bb=80 150 800 575]{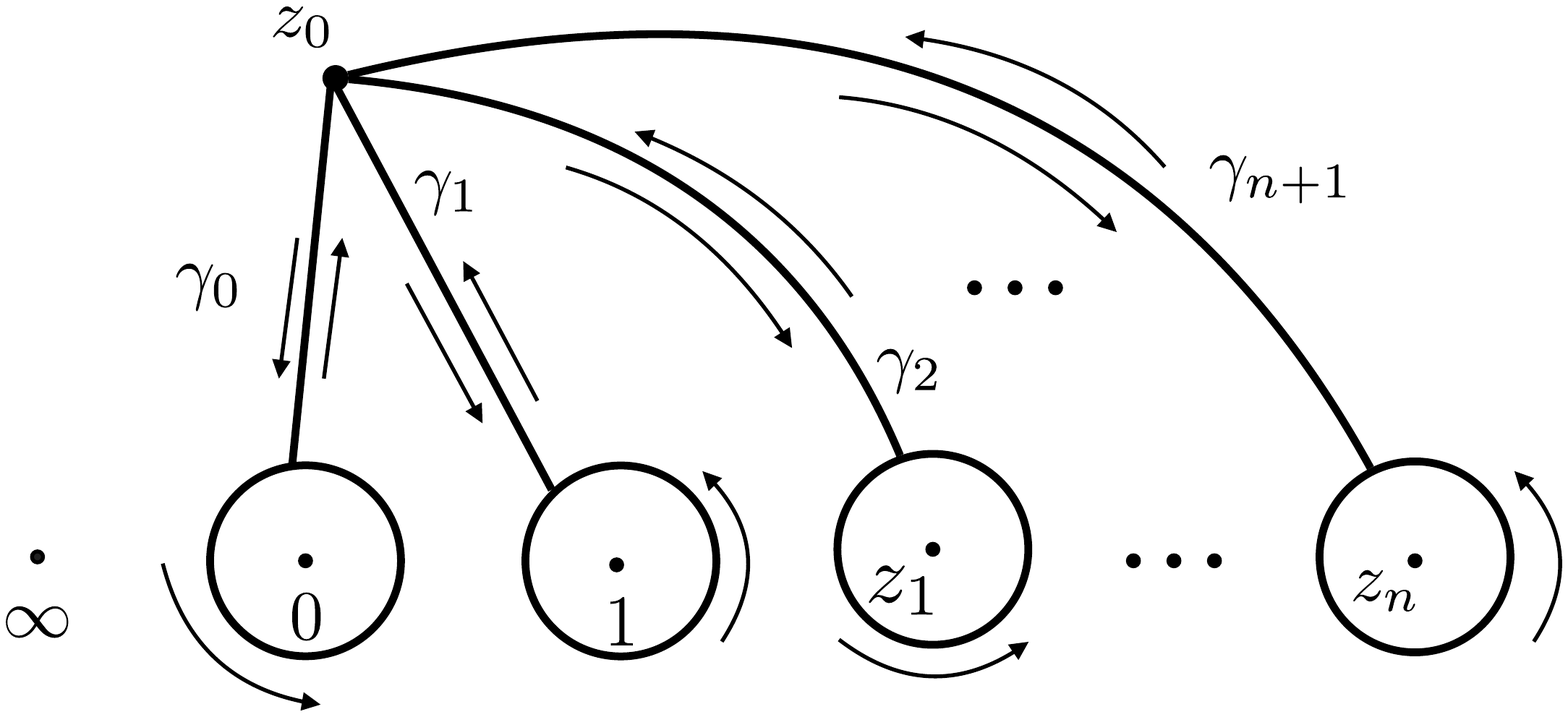}
\caption{}
\label{FigurePB}
\end{figure}

Obviously, each of them represents a non-trivial element in $F_{n+2}$.
Then, we put
\begin{equation}
	\label{eqn:curve-gamma}
	\gamma=[\beta, \tilde{\gamma}_{n+1}].
\end{equation}

The curve $\gamma$ also represents a non-trivial element in $F_{n+2}$. 
Indeed, from the definition we have
\begin{eqnarray*}
	\gamma&=&\beta\tilde{\gamma}_{n+1}\beta^{-1}\tilde{\gamma}_{n+1}^{-1}=\gamma_0\gamma_1\dots\gamma_{n+1}\tilde{\gamma}_{n+1}\gamma_{n+1}^{-1}\dots\gamma_{1}^{-1} \gamma_{0}^{-1}\tilde{\gamma}_{n+1}^{-1}\\
	&=& \gamma_0\gamma_1\dots\gamma_{n+1}\gamma_{n+1}\tilde{\gamma}_{n}\gamma_{n+1}^{-1}\tilde{\gamma}_{n}^{-1}\gamma_{n+1}^{-1}\dots\gamma_{1}^{-1}\gamma_{0}^{-1}\tilde{\gamma}_{n}\gamma_{n+1}\tilde{\gamma}_{n}^{-1}\gamma_{n+1}^{-1}.
\end{eqnarray*}
Since $\tilde{\gamma}_n=\gamma_n\tilde{\gamma}_{n-1}\gamma_{n}^{-1}\tilde{\gamma}_{n-1}^{-1}$, there is not a reduction in $\gamma_0^{-1}\tilde{\gamma}_n$ appearing at the last 3rd and 4th positions in the above word of $\gamma$.
Hence, we verify that $\gamma$ represents a non-trivial element in $F_{n+2}$.

If we remove all $\gamma_{j}^{\pm}$ for the word of $\gamma$, then we have the trivial element since $\tilde{\gamma}_j$ becomes trivial and all of $\tilde{\gamma}_{i}$ $(i=j+1, \dots, n+1)$ become trivial.
We see that $\gamma$ has the following property:
\begin{description}
  \item[(A)] if we remove $\gamma_{j}^{\pm 1}$ from the word of $\gamma$ $(j=0, 1, 2, \dots$ or $n+1)$, then we obtain the trivial element.
\end{description}

Now, we consider a Fuchsian group $\Gamma_n$ acting on $\Delta$ such that $\Delta/\Gamma_n=\mathbb{C}\setminus\{0, 1, z_1, \dots , z_n\}$, which is isomorphic to $F_{n+2}$.
Take $g\in \Gamma_n$ which corresponds to $\gamma$.
The quotient space $\Delta/<g>$ is an annulus $A=\{1<|\lambda|<R\}$ for some $R>1$.
We denote by $\pi_{g} : A\to \mathbb{C}\setminus\{0, 1, z_1, \dots , z_n\}$ the canonical projection and take $\lambda_0\in A$ with $\pi_{g}(\lambda_0)=z_0$ as a basepoint.
Take a circle $\alpha :=\{|\lambda|=|\lambda_0|\}\subset A$.

We define a holomorphic motion $\phi : A\times E\to \widehat{\mathbb{C}}$ over the Riemann surface $A$ by
\begin{equation*}
	\phi (\lambda , z)=\begin{cases}
		z, \quad &(z=0, 1, z_1, \dots , z_n) \\
		\pi_{g}(\lambda), \quad &(z=z_0).
	\end{cases}
\end{equation*}
Since the curve $\gamma$ represents a non-trivial element in $F_{n+2}$, 
we verify that the monodromy of $\phi$ for $\alpha$ is not trivial because of the same reason as in (1).
Therefore, $\phi$ cannot be extended to a holomorphic motion of $\widehat{\mathbb{C}}$ over $A$.

Let $E'$ be a proper subset of $E$.
We see that $\phi|E'$ is extended to a holomorphic motion of $E'$ over $A$.
It suffices to show that the monodromy of $\phi|{E'}$ is trivial.

If $E'$ does not contain $z_0$, then the monodromy of $\phi |{E'}$ is trivial because $\phi (\lambda , z)=z$ for any $\lambda \in A$ and for any $z\in E'$.

Suppose that $E'$ contains $z_0$.
Since $E'$ consists of at most $(n+3)$ points containing $z_0$, it does not contain at least one point in $\{0, 1, \infty, z_1, \dots , z_{n}\}$. 

If $E'\not \ni \infty$, then the curve $\beta$ is trivial because it surrounds $\infty$.
Hence, $\gamma (=[\beta, \tilde{\gamma}_{n+1}])$ represents the trivial element of $\pi_{1}(\widehat{\mathbb{C}}\setminus(E'\setminus\{z_0\}), z_0)$.

If $E'\not\ni z\in \{0, 1, z_1, \dots, z_{n+1}\}$, then some $\gamma_j$ surrounding $z$ becomes trivial.
Hence, from the property (A), we verify that the curve $\gamma$ also represents the trivial element of $\pi_{1}(\widehat{\mathbb{C}}\setminus(E'\setminus\{z_0\}), z_0)$.

Thus, in any case, we see that $\gamma$ represents the trivial element of $\pi_{1}(\widehat{\mathbb{C}}\setminus(E'\setminus\{z_0\}), z_0)$;
a closed curve $\gamma_{\alpha}:=\phi(\alpha, z_0)$ is homotopic to the trivial curve in $\pi_{1}(\widehat{\mathbb{C}}\setminus(E'\setminus\{z_0\}), z_0)$ since $\gamma_{\alpha}$ is homotopic to $\gamma$ from the construction of the Riemann surface $A$.
From Lemma \ref{lemma:Kra}, we verify that the monodromy of $\phi|E'$ is trivial for any $E'$.

\section{Proof of Theorem \ref{thm:mythm2}}
We may obtain Theorem \ref{thm:mythm2} from Theorem \ref{thm:mythm1}. It is done by following the argument of Earle-Kra-Krushkal \cite{Earle1994}.
For readers' convenience, we will give a sketch of the proof.

Let $G$ be a subgroup of $\textrm{M\"ob}(\mathbb{C})$ and $E$ be a $G$-invariant subgroup of $\widehat{\mathbb{C}}$.
As in Theorem \ref{thm:mythm1}, we to show that if the monodromy of a $G$-equivariant holomorphic motion $\phi$ of $E$ over $X$ is trivial, then it is extended to a $G$-equivariant holomorphic motion of $\widehat{\mathbb{C}}$ over $X$.

For simplicity, we assume that $G$ is torsion free. 
First of all, we may assume that $E$ is a closed subset of $\widehat{\mathbb{C}}$ because of the $\lambda$-lemma. 
Hence, $E$ contains the set of all fixed points of $G$ since a fixed point of any $g\in G$ is either an attracting or a repelling fixed point of $g$. 

Let $\phi : X\times  E\to \widehat{\mathbb{C}}$ be a $G$-equivariant holomorphic motion of $E$ over $X$ satisfying (\ref{eqn:equivariance}).
Then, since $\{\theta_{p}\}_{p\in X}$ is a holomorphic family of isomorphisms of $G$ obtained by quasiconformal maps, every $\theta_{p}$ is a type-preserving isomorphism.

Take a point $z_0$ in $E^{c}$. Then, there exists a holomorphic motion $\phi_{0} : X\times (E\cup\{z_0\})\to \widehat{\mathbb{C}}$ which extends $\phi$.
Indeed, it follows from Theorem \ref{thm:mythm1} that we have a holomorphic motion of $\widehat{\mathbb{C}}$ over $X$ which extends $\phi$. By restricting the holomorphic motion to $E\cup\{z_0\}$, we have the desired holomorphic motion $\phi_0$.

Now, we define a map $\widehat{\phi}_0$ on $X\times (E\cup G(z_0))$ by
\begin{equation}
	\label{eqn:equivExtent}
	\widehat{\phi}_0 (p, z)=
	\begin{cases}
		\phi_0 (p, z), \quad z\in E\cup\{z_0\} \\
		\theta_{p}(g)(\phi _{0}(p, z_0)), \quad z=g(z_0).
	\end{cases}
\end{equation}
We will show that $\widehat{\phi}_{0}$ is a $G$-equivariant holomorphic motion of $E\cup G(z_0)$.

Since $\theta_{p_o}(g)=g$ for $g\in G$, we have $\widehat{\phi}_{0}(p_0, \cdot )=id$ on $E\cup G(z_0)$. 
It is also obvious that $\widehat{\phi}_{0}(\cdot, z)$ is holomorphic on $X$. We show the injectivity of $\widehat{\phi}_{0}(p, \cdot )$.

From the $G$-invariance of $E$, we see that $E\cap G(z_0)=\emptyset$ and $\widehat{\phi}_0 (p, \cdot)$ is injective on $E$.
Suppose that $\widehat{\phi}_{0}(p, z)=\widehat{\phi}_{0}(p, g(z_0)$ for some $z\in E$ and $g\in G$. Then we have
\begin{equation}
	\label{eqn:equiInj}
	\phi _{0}(p, z)=\theta_{p}(g)(\phi_{0}(p, z_0)).
\end{equation}
However, from the $G$-invariance of $E$ we have
\begin{equation*}
	\phi_{0}(p, E)=\phi_{0}(p, G(E))=\theta_{p}(G)(\phi_{0}(p, E)).
\end{equation*}
Hence, $\phi_{0}(p, E)$ is $\theta (G)$ invariant. 
Since $\phi_0 (p, z_0)$ is not in $\phi_0 (p, E)$, it contradicts (\ref{eqn:equiInj}).

Finally, suppose that $\widehat{\phi}_{0}(p, g_1 (z_0))=\widehat{\phi}_{0}(p, g_2 (z_0))$ for $g_1, g_2 \in G$.
Then, we have
\begin{equation*}
	\theta_{p}(g_1)(\phi_{0}(p, z_0))=\theta_{p}(g_2)(\phi_{0}(p, z_0)).
\end{equation*}
If $g_1\not=g_2$, then $\phi_{0}(p, z_0)$ is a fixed point of $\theta_{p}(g_1^{-1}\circ g_2)$.
However, this implies that $\phi_{0}(p, z_0)$ is contained in $\phi_{0}(p, E)$ which is $\phi (p, E)$, and we have a contradiction since $\phi_{0}(p, \cdot)$ is injective on $E\cup\{z_0\}$.
Thus, we have shown that $\widehat{\phi}_{0}$ is a holomorphic motion of $E\cup G(z_0)$. 
The $G$-equivariance of $\widehat{\phi}_0$ is trivial from (\ref{eqn:equivExtent}).

By repeating this procedure, we obtain a $G$-equivariant holomorphic motion $\widehat{\phi}_{\infty}$ of a countable dense subset $E_{\infty}$ of $\widehat{\mathbb{C}}$ over $X$.
It follows from the $\lambda$-lemma that $\widehat{\phi}_{\infty}$ is extended to a holomorphic motion of $\widehat{\mathbb{C}}$, the closure of $E_{\infty}$. 
It is also easy to verify that the extended holomorphic motion is $G$-equivariant. Thus, we have completed the proof of Theorem \ref{thm:mythm2}.

\section{Applications.}
\label{sec:Appl}

\subsection{Topological conditions for the extendability of holomorphic motions}
A map $\phi : M\times E\to \widehat{\mathbb{C}}$ is called a \emph{continuous motion} of $E$ over $M$ if it satisfies the conditions (1), (2) in Definition \ref{dfn:holomorphic motion} and

(3') $\phi$ is continuous in $M\times E$ and for each $p\in M$, $\phi (p, \cdot)$ is a homeomorphism from $E$ onto its image. 

\medskip

The concept of the monodromy of holomorphic motions is topological. Hence, immediately we have the following.
\begin{Mythm}
\label{thm:mythem3}
	Let $\phi : X\times E\to \widehat{\mathbb{C}}$ be a holomorphic motion of $E$ over a Riemann surface $X$.
	Suppose that $\phi$ is extended to a continuous motion of $\widehat{\mathbb{C}}$ over $X$. Then, 
	\begin{enumerate}
  \item the holomorphic motion $\phi$ is extended to a holomorphic motion of $\widehat{\mathbb{C}}$ over $X$;
  \item if $\phi$ is $G$-equivariant for a subgroup $G$ of $\textrm{M\"ob}(\mathbb{C})$, then it is extended to a $G$-equivariant holomorphic motion of $\widehat{\mathbb{C}}$ over $X$.
\end{enumerate}
\end{Mythm}
\begin{Rem}
\begin{enumerate}
  \item In the second statement, we do not assume that the extended continuous motion is $G$-equivariant.
  \item Gardiner, Jiang and Wang \cite{Gardiner2015a} announce that if a holomorphic motion $\phi : X\times E\to \widehat{\mathbb{C}}$ has a guiding quasiconformal isotopy, then $\phi$ is extended to a holomorphic motion of $\widehat{\mathbb{C}}$ over $X$ which extends $\phi$. 
A guiding isotopy is a continuous motion with a quasiconformal nature (see \cite{Gardiner2015a} for the precise definition). Hence, Theorem \ref{thm:mythem3} confirms their result.
\end{enumerate}
\end{Rem}

Now, we consider the triviality of the monodromy. 
In general, it is not easy to see if the monodromy of a holomorphic motion is trivial or not.
However, under a certain condition the monodromy becomes automatically trivial (cf.~\cite{BJMS2012} Lemma 8.1).
\begin{Pro}
	\label{pro:simplyconnectedCase}
	Let $E\subset \widehat{\mathbb{C}}$ be a closed set and $w :\widehat{\mathbb{C}}\to \widehat{\mathbb{C}}$ be a quasiconformal homeomorphism fixing each point of $E$.
	Suppose that $E$ is connected.
	Then, $w$ is homotopic to the identity rel $E$.
\end{Pro}
A discrete subgroup of $\textrm{M\"ob}(\mathbb{C})$ is called a Kleinian group.
The limit set $\Lambda(G)$ of a Kleinian group $G$ is defined by the closure of the set of fixed points of elements of $G$ with infinite order.
The limit set $\Lambda(G)$ of $G$ is closed and invariant under the action of $G$.
A Kleinian group $G$ is called \emph{non-elementary} if the limit set $\Lambda(G)$ contains more than two points.

From Theorems \ref{thm:mythm1}, \ref{thm:mythm2} and Proposition \ref{pro:simplyconnectedCase}, immediately we have
\begin{Cor}
\label{Cor:connected case}
	Let $\phi : X\times E\to \widehat{\mathbb{C}}$ be a holomorphic motion of $E$ over a Riemann surface $X$.
	Suppose that $E$ is connected. 
	Then, $\phi$ can be extended to a holomorphic motion of $\widehat{\mathbb{C}}$ over $X$.
	
	Furthermore, if the holomorphic motion $\phi$ is $G$-equivariant for  a subgroup $G$ of $\textrm{M\"ob}(\mathbb C)$, then
	the holomorphic motion can be extended to a $G$-equivariant holomorphic motion of $\widehat{\mathbb{C}}$ over $X$.
\end{Cor}
\begin{Rem}
	If the limit set $\Lambda(G)$ of a non-elementary Kleinian group $G$ is not connected, then there is a $G$-equivariant holomorphic motion of $\Lambda(G)$ over the punctured disk $\Delta^{*}:=\{0<|z|<1\}$ which cannot be extended to a holomorphic motion of $\widehat{\mathbb{C}}$ over $\Delta^{*}$ (\cite{Shiga2016} Theorem II).
	
\end{Rem}

\subsection{A geometric condition for the extendability of holomorphic motions}
Let $E=\{0, 1, \infty, z_1, \dots, z_{n}\}$ be a finite subset of $(n+3)$ points in $\widehat{\mathbb{C}}$.
On the Riemann surface $S_{E}:=\widehat{\mathbb{C}}\setminus E$, we may define the hyperbolic metric which is the projection of the hyperbolic metric ${(1-|z|^2)^{-1}}|dz|$ of the unit disk, the universal covering of $S_E$. 

For the set $E$, we consider the following quantity:
\begin{equation}
	\label{dfn:minimal Length in Config}
	L(E)=\min \left\{\log (2+\sqrt{5}), \frac{1}{2}\log \left (\left (\frac{\ell(E)}{\pi}\right )^2+1\right )\right \},
\end{equation} 
where $\ell (E)$ is the minimal length of non-trivial and non-peripheral closed curves in $S_E$ with respect to the hyperbolic metric on $S_E$.
Then, we have the following.
\begin{Mythm}
\label{thm:mythm4}
	Let $X$ be a hyperbolic Riemann surface with a basepoint $p_0$. 
	Suppose that the fundamental group $\pi_1 (X, p_0)$ is of finitely generated and there exist closed curves $\gamma_1, \gamma_2, \dots , \gamma_m$ passing though $p_0$ such that the homotopy classes of those curves generate $\pi_1 (X, p_0)$ and
\begin{equation}
\label{eqn:shortLength}
	\ell (\gamma_j) < L(E) \quad (j=1, 2, \dots, m)
\end{equation}
hold, where $\ell (\gamma)$ is the hyperbolic length of a curve $\gamma\subset X$.
Then, every holomorphic motion $\phi : X\times E\to \widehat{\mathbb{C}}$ of $E$ over $X$ can be extended to a holomorphic motion of $\widehat{\mathbb{C}}$ over $E$.
\end{Mythm}
\begin{proof}
	Let $\phi : X\times E\to \widehat{\mathbb{C}}$ be a holomorphic motion of $E$ over a Riemann surface $X$ satisfying (\ref{eqn:shortLength}).
	From Theorem \ref{thm:mythm1}, it suffices to show that the monodromy of $\phi$ is trivial.
	Since the triviality of the monodromy is the triviality of a homomorphism from $\pi_1 (X, p_0)$ to $\textrm{Mod}(0, n+3)$, it is enough to show that each $\gamma_j$ $(j=1, 2, \dots , m)$ gives the identity of $\textrm{Mod}(0, n+3)$. 
	To show it, we consider the action of the pure mapping class group, which is a subgroup of $\textrm{Mod}(0, n+3)$, on the Teichm\"uller space $Teich(S_{E})$.
	
	Let $P\textrm{Mod}(0, n+3)$ denote a subgroup of mapping classes in $\textrm{Mod}(0, n+3)$ whose representatives are quasiconformal maps of $\widehat{\mathbb{C}}$ fixing each point of $E$.
	The subgroup $P\textrm{Mod}(0, n+3)$ is called {\em the pure mapping class group}.
	We consider the infimum of translation lengths of $P\textrm{Mod}(0, n+3)$ of $P_0:=[S_E, id]\in Teich(S_E)$. That is,
	\begin{equation}
	\label{Dfn:InjRadius}
		r_{E}=\inf \{d_{T}(P_0, \chi (P_0)) : \chi\in P\textrm{Mod}(0, n+3)\setminus\{id\}\},
	\end{equation}
	where $d_T$ is the Teichm\"uller distance defined in \S 2.
	Then, we have shown the following.
	\begin{Pro}[ (\cite{Shiga2013} Theorem 2.1) ]
	\label{Pro:configulationSpace}
		$$
		L(E)\leq r_{E}.
		$$
	\end{Pro}
	
	Let $\Gamma$ be a Fuchsian group on $\Delta$ which represents $X$.
	For the natural projection $\pi : \Delta\to \Delta/\Gamma=X$, the map $\tilde{\phi}(\lambda , z):=\phi (\pi (\lambda ), z)$ $(\lambda\in \Delta, z\in E)$ defines a holomorphic motion of $E$ over $\Delta$.
	We see from Proposition \ref{Pro:universal} that there exists a holomorphic map $\Phi : \Delta \to Teich(S_E)$ such that $\Phi (\lambda)$ gives $\widehat{\mathbb{C}}\setminus \tilde{\phi}(\pi (\lambda), E)$ as a point in $Teich(S_E)$ for any $\lambda\in \Delta$.
	
	We may take $0\in \Delta$ so that $\pi (0)=p_0\in X$.
	Then, we have $\Phi (0)=P_0$. 
	Let $\tilde{\gamma}_j$ be a lift of $\gamma_j$ on $\Delta$ which begins at $0$ $(j=1, 2, \dots, m)$.
	Then, there exists $g_j\in \Gamma$ such that $g_j (0)$ is the end point of $\tilde{\gamma}_j$.
	Since $\pi (g_j (0))=\pi (0)=p_0$, we have $\tilde{\phi}(g_j (0), z)=\tilde{\phi}(0, z)=z$ for any $z\in E$.
	Hence, there exists $\chi_j\in P\textrm{Mod}(0, n+3)$ such that
	\begin{equation*}
		\Phi (g_j (0))=\chi_j (\Phi (0))=\chi_j (P_0).
	\end{equation*}	
	The Teichm\"uller distance is the Kobayashi distance. It follows from the distance decreasing property of the Kobayashi distance (cf. \cite{Kobayashi1970}) that
	\begin{equation}
	\label{eqn:distanceDecreasing}
		d_T (P_0, \chi_j (P_0))= d_T (\Phi (0), \Phi (g_j (0)))\leq d_{\Delta}(0, g_j (0)),
	\end{equation} 
	where $d_{\Delta}$ is the hyperbolic distance in $\Delta$, which is the Kobayashi distance of $\Delta$.

	On the other hand, since $\tilde{\gamma}_j$ is an arc from $0$ to $g_j(0)$, we have 
	\begin{equation}
	\label{eqn:LengthIneq}
		\ell (\gamma_j)\geq d_{\Delta}(0, g_j (0)).
	\end{equation}
	Combining (\ref{eqn:shortLength}), (\ref{eqn:distanceDecreasing}), (\ref{eqn:LengthIneq}) and Proposition \ref{Pro:configulationSpace}, we have
	\begin{equation*}
		d_T (P_0, \chi_j (P_0))\leq \ell (\gamma_j) <L(E)\leq r_{E}
	\end{equation*}
	for each $j$ $(1\leq j\leq m)$. 
	Therefore, we conclude that $\chi_j=id$ for every $j$ $(j=1, 2, \dots , m)$ because of (\ref{Dfn:InjRadius}).
	Thus, we have proved that the monodromy of the holomorphic motion $\phi$ is trivial.
	
\end{proof}
Let $A$ be an annulus $\{1<|\lambda|<R\}$ $(R>1)$. Then, the curve $\alpha:= \{|\lambda|=\sqrt{R}\}$ is the shortest closed curve with respect to the hyperbolic metric on $A$ which generates the fundamental group of $A$. 
It is not hard to see that the hyperbolic length of $\alpha$ is $\pi^2/\log R$.
Hence, from Theorem \ref{thm:mythm4} we have the following.
\begin{Cor}
Let $E$ be a finite set in $\widehat{\mathbb{C}}$ and $A=\{1<|\lambda|<R\}$ be an annulus with a basepoint $p_0=\sqrt{R}$. Suppose that
$$
R>\exp \left \{\frac{\pi^2}{L(E)}\right \}.
$$	
Then, any holomorphic motion of $E$ over $A$ is extended to a holomorphic motion of $\widehat{\mathbb{C}}$ over $A$.
\end{Cor}

\subsection{Lifting problems}
\label{subsec:lifting}
As we have seen in \S \ref{subsection:Teich}, there is a surjective map from the space of Beltrami coefficients to the Teichm\"uller space. 
The Douady-Earle section in \S \ref{subsection:Douady-Earle} gives the inverse map of the surjective map; however, the section is not holomorphic.

It is known that there exists no holomorphic inverses on the Teichm\"uller space. We show that any holomorphic map from a Riemann surface to the Teichm\"uller space can be lifted to a holomorphic map to the space of Beltrami coefficients. We consider the problem for two kinds of Teichm\"uller spaces, Teichm\"uller space of a Riemann surface and that of a closed set, separately.

{\bf Teichm\"uller space of a Riemann surface.}
Let $S$ be a hyperbolic Riemann surface represented by a Fuchsian group $\Gamma$ on $\Delta$. Let $\pi_{\Gamma} : \textrm{Belt}(\Gamma)\to Teich(S)$ be the holomorphic projection defined in \S \ref{subsection:Teich}. Then, we have the following.

\begin{Mythm}
	\label{Mythm:SectionRiemann}
	Let $F : X\to Teich (S)$ be a holomorphic map from a Riemann surface $X$ to the Teichm\"uller space $Teich (S)$ of a Riemann surface $S$ represented by a Fuchsian group $\Gamma$ on $\Delta$.
	Then, there exists a holomorphic map $\tilde{F}$ from $X$ to $\textrm{Belt}(\Gamma)$ which satisfies the following commutative diagram.
		\begin{equation*}
		\xymatrix{
		 & \textrm{Belt}(\Gamma) \ar[d]^{\pi_{\Gamma}} \\
		 X \ar[ru]^{\widetilde{F}}\ar[r]_{F} &  Teic(S)		 }
	\end{equation*}
\end{Mythm}
\begin{proof}
	To prove the theorem, we introduce the Bers embedding of $Teich(S)$ (cf. \cite{Imayosh1992}).
	
	For each $\mu\in \textrm{Belt}(\Gamma)$, we put
	\begin{equation}
	\label{dfn:tildemu}
		\tilde\mu (z)=\begin{cases}
			\mu (z), \quad &(z\in \Delta) \\
			0, \quad &(z \in \Delta^c).
		\end{cases}
	\end{equation}
	Then, the function $\tilde \mu$ belongs to $M(\mathbb C)$, the space of Beltrami coefficients on $\mathbb C$. 
	Moreover, it is $\Gamma$-compatible, namely, it satisfies
	\begin{equation*}
\tilde \mu (\gamma (z))\frac{\overline{\gamma '(z)}}{\gamma '(z)}=\tilde \mu (z) 
	\end{equation*}
	almost everywhere in $\mathbb C$.
	
	We take a quasiconformal map $f^{\mu}$ on $\widehat{\mathbb C}$ for $\tilde \mu$. The map is a solution of the Beltrami equation
	\begin{equation*}
		\overline\partial f=\tilde \mu \partial f
	\end{equation*}
	in $\mathbb C$.
	From the definition (\ref{dfn:tildemu}), the map $f^{\mu}$ is conformal in ${\overline{\Delta}}^{c}$. We normalize $f^{\mu}$ by
	\begin{equation*}
		f^{\mu}(z)=z+ O\left (\frac{1}{z}\right )
	\end{equation*}
	as $z\to \infty$.
	We see that $f^{\mu}$ is uniquely determined by $\mu$.
	
	Since $\tilde \mu$ is $\Gamma$-compatible, there exists an isomorphism $\Theta_{\mu} : \Gamma\to \textrm{M\"ob}(\mathbb C)$ such that
	\begin{equation}
	\label{eqn:eqivf^mu}
		f^{\mu}(\gamma (z))=\Theta_{\mu}(\gamma)(f^{\mu}(z))
	\end{equation} 
	holds for every $\gamma\in \Gamma$ and for every $z\in \mathbb{C}$. 
	It is known that  $f^{\mu}$ and the isomorphism $\Theta_{\mu}$ depend only on $[\mu]_{\Gamma}\in Teich(S)$.
	Therefore, we may denote them by $f^{[\mu]_{\Gamma}}$ and $\Theta_{[\mu]_{\Gamma}}$, respectively.
	The group $\Gamma_{[\mu]_{\Gamma}}:= \Theta_{[\mu_{\Gamma}]}(\Gamma)$ is a quasiconformal conjugate to $\Gamma$ called a quasi-Fuchsian group and its limit set $\Lambda(\Gamma_{[\mu]_{\Gamma}})$ is a Jordan curve.
	
	Taking the Schwarzian derivative $\{f^{[\mu]_{\Gamma}}, \cdot\}$ of $f^{[\mu]_{\Gamma}}$ in ${\overline{\Delta}}^{c}$, we have from (\ref{eqn:eqivf^mu})
	\begin{equation*}
		\{f^{[\mu]_{\Gamma}}, \gamma (z)\}\gamma'(z)^2=\{f^{[\mu]_{\Gamma}}, z\}\quad (z\in {\overline{\Delta}}^{c}).
	\end{equation*}
	The above equation shows that the Schwarzian derivative $\{f^{[\mu]_{\Gamma}}, \cdot\}$ is a holomorphic quadratic differential for $\Gamma$; it is bounded, i. e.,
	\begin{equation*}
		\sup_{z\in {\overline{\Delta}}^{c}}\rho_{{\overline{\Delta}}^{c}}^{-2}|\{f^{[\mu]_{\Gamma}}, z\}|<+\infty,
	\end{equation*}
	where $\rho_{{\overline{\Delta}}^{c}}(z)|dz|$ is the Poincar\'e metric on ${\overline{\Delta}}^{c}$. 
	Since the Schwarzian derivative depends only on $[\mu]_{\Gamma}(=\pi_{\Gamma}(\mu))$, we may denote it by $\beta ([\mu]_{\Gamma})$.
	
	Thus, we have a map $\beta$ on the Teichm\"uller space $Teich(S)$ of $S$ to the space of bounded holomorphic quadratic differentials on ${\overline{\Delta}}^{c}$ for $\Gamma$.
	It is known that the map $\beta$ is injective and holomorphic on $Teich(S)$ and it is called the {\em Bers embedding} of $Teich(S)$.
	
	Now, we consider a holomorphic map $\Phi:=\beta\circ F$ on $X$.
	Then, the conformal maps $f^{\Phi (\lambda)}$ $(\lambda\in X)$, which are solutions of $\{f, z\}=\Phi (\lambda)(z)$ on ${\overline{\Delta}}^{c}$, depend holomorphically on $\lambda\in X$.
	Therefore, we have a holomorphic motion $\phi$ of ${\overline{\Delta}}^{c}$ over $X$ defined by
	\begin{equation}
		\phi (\lambda, z)=f^{\Phi (\lambda)}(z),\quad (\lambda, z)\in X\times {\overline{\Delta}}^{c}.
	\end{equation}
	
	It follows from the $\lambda$-lemma that the holomorphic motion is extended to a holomorphic motion of $\Delta ^c$ over $X$.
	The extended holomorphic motion is $\Gamma$-equivariant because of (\ref{eqn:eqivf^mu}). Since $\Delta^c$ is connected, there exists a $\Gamma$-equivariant holomorphic motion $\widehat \phi$ of $\widehat{\mathbb C}$ over $X$ which extends $\phi$ by Corollary \ref{Cor:connected case}.
	Define 
	\begin{equation*}
		\widetilde F (\lambda)=\frac{\overline \partial \widehat \phi (\lambda, \cdot)}{\partial \widehat \phi (\lambda, \cdot)} \in \textrm{Belt}(\Gamma),\quad (\lambda\in X)
	\end{equation*}
	then it is a lift of $F$ from the construction. Furthermore, it is holomorphic by Theorem 2 in \cite{Earle1994}; we obtain a holomorphic lift $\widetilde F$ of $F$ as desired.

\end{proof}

{\bf Teichm\"uller space of a closed set.}
Let $E$ be a closed set in $\widehat{\mathbb C}$. We consider the lifting problem for $T(E)$.
First, we note a relationship between a holomorphic maps $F : X\to T(E)$ and holomorphic motions with trivial monodromy.

Let $\phi : X\times E\to \widehat{\mathbb{C}}$ be a normalized holomorphic motion of $E$ over a Riemann surface $X$.
We may assume that the universal covering $\widetilde{X}$ of $X$ is conformally equivalent to the unit disk $\Delta$ unless the holomorphic motion is trivial.

Indeed, if $\widetilde{X}$ is conformally equivalent to the Riemann sphere or the complex plane, then $\widetilde{X}$ does not admit a non-constant bounded holomorphic function on $\widetilde{X}$. 
It follows from Theorem 1 in \cite{Mitra2010} that $\phi\circ\pi$ is a trivial holomorphic motion of $E$ over $\widetilde{X}$, where $\pi : \widetilde{X}\to X$ is a universal covering map and $\phi\circ\pi$ is a holomorphic motion of $E$ over $\widetilde{X}$ defined by $\phi\circ\pi (p, z)=\phi (\pi (p), z)$ for $(p, z)\in \widetilde{X}\times E$.
Since $\phi$ is a normalized holomorphic motion, we have $\phi (\cdot, z)\equiv z$ for any $z\in E$.
Hence, $\phi$ is extended to a holomorphic motion trivially.

Let $\Gamma$ be a Fuchsian group acting on $\Delta$ such that $\Delta/\Gamma\simeq X$. 
Then, $\phi\circ\pi$ defined as above is a holomorphic motion of $E$ over $\Delta$.
From the universal property of the Teichm\"uller space $T(E)$, we have a holomorphic map $f :\Delta\to T(E)$ which induces $\phi\circ\pi$ (Proposition \ref{Pro:universal}).
If the monodromy of $\phi$ is trivial, then we verify that $f\circ\gamma=f$ for any $\gamma\in\Gamma$. 
Therefore, $f$ defines a holomorphic map $F$ from $X$ to $T(E)$.

Conversely, any holomorphic map $F : X\to T(E)$ gives a holomorphic motion $\phi_{F} : X\times E\to \widehat{\mathbb{C}}$ by
\begin{equation*}
	\phi_{F}(p, z)=\Psi (F(p), z), \quad (p, z)\in X\times E.
\end{equation*}
We also see that the monodromy of $\phi_{F}$ is trivial from the definition of $T(E)$.

We have constructed a projection $P_E$ from $M(\mathbb C)$ onto $T(E)$ and the real analytic section of $P_E$ in \S \ref{subsection:Douady-Earle}.  
On the other hand, there are no global holomorphic sections of $P_E$ in general.
The following theorem, however, implies that a holomorphic map from a Riemann surface to $T(E)$ can be lifted to $M(\mathbb C)$ via $P_E$. 
\begin{Mythm}
	\label{thm:sectionX}
	Let $F : X\to T(E)$ be a holomorphic map of a Riemann surface $X$ to Teichm\"uller space $T(E)$ of a closed set $E\subset\widehat{\mathbb{C}}$. Then, there exists a holomorphic map $\widetilde{F}$ from $X$ to $M(\mathbb{C})$, the space of Beltrami coefficients on $\mathbb{C}$, which satisfies the following commutative diagram.
	\begin{equation*}
		\xymatrix{
		 & M(\mathbb{C}) \ar[d]^{P_E} \\
		 X \ar[ru]^{\widetilde{F}}\ar[r]_{F} &  T(E)		 }
	\end{equation*}
\end{Mythm}
\begin{proof}
	As we have seen above, the holomorphic map $F$ gives a holomorphic motion $\phi_{F}$ of $E$ over $X$ with trivial monodromy.
	Therefore, we have a holomorphic motion $\hat{\phi}_{F} : X\times\widehat{\mathbb{C}}\to\widehat{\mathbb{C}}$ which extends $\phi_{F}$ by Theorem \ref{thm:mythm1}.
	Hence, we have a map $\widetilde{F} : X\to M(\mathbb{C})$ by
	\begin{equation*}
		\widetilde{F}(p)=\frac{\overline{\partial}\hat{\phi}_{F}(p, \cdot)}{\partial\hat{\phi}(p, \cdot)}\in M(\mathbb{C}).
	\end{equation*}
	We see that it is a lift of $F$.
	Moreover, the map $\widetilde{F}$ is holomorphic on $X$ by Theorem 2 in \cite{Earle1994}.
	Hence, $\widetilde F$ is a holomorphic lift of $F$.
\end{proof}

%
%
%


\begin{thebibliography}{99}
  \bibitem {Ahlfors-Sario} L. V. Ahlfors and L. Sario, Riemann Surfaces, Princeton University Press, 1960.
       \bibitem {BJMS2012} M. Beck, Y. Jiang, S. Mitra and H. Shiga, 
       Extending holomorphic motions and monodromy, Ann. Acad. Sci. Fenn. 
       {\bf 37} (2012), 53--67.
       \bibitem {Bers1986a} L. Bers and H. L. Royden, Holomorphic families of injections, Acta Math. {\bf 157} (1986), no. 1, 259--286.
       \bibitem {Chirka2004} E. M. Chirka, On the extension of holomorphic motions, Dokl. Math. {\bf 70} (2004), no. 1, 37--40.
       \bibitem {Dieudonn1969} J. Dieudonn\'e, Foundations of Modern Analysis, Academic Press, New York and London, 1969.
       \bibitem {Douady1986} A. Douady and C. J. Earle, Conformally natural extension of homeomorphisms of the circle, Acta Math. {\bf 157} (1986), no. 1, 23--48.
       \bibitem {Duren1970} P. Duren, Theory of $\textrm{H}^{p}$-Spaces, Academic Press, New York, 1970.
       \bibitem{Earle1994} C. J. Earle, I. Kra, and S. L. Krushkal, Holomorphic motions and Teichm\"uller spaces, Trans. Amer. Math. Soc. {\bf 343} (1994), no. 2, 927--948.
       \bibitem {Earle2000} C. J. Earle and S. Mitra, Variation of moduli under holomorphic motions, Contemp. Math. {\bf 256} (2000), 39--67.
       \bibitem{ EarlePre}  C. J. Earle and S. Mitra, Real analyticity and conformal naturality of barycentric sections in Teichm\"uller spaces, in preparation. 
       \bibitem {Gardiner2015a} F. P. Gardiner, Y. Jiang, and Z. Wang, Guiding isotopies and holomorphic motions, Ann. Acad. Sci. Fenn. Ser. AI Math {\bf 40} (2015), 485--501.
       \bibitem {Imayosh1992} Y. Imayoshi and M. Taniguchi, {\it Introduction to Teichm\"uller spaces}, Springer, Tokyo, 1992.
       \bibitem {Kassel2008} C. Kassel and V. Turaev, {\it {Braid Groups}}, Springer, 2008.
       \bibitem {Kobayashi1970} S. Kobayashi, Hyperbolic Manifolds and Holomorphic Mappings, Marcel Dekker, 1970.
       \bibitem {Kra1981} I. Kra, On the Nielsen-Thurston-Bers type of some self-maps of Riemann surfaces, Acta Math. {\bf 146} (1981), no. 1, 231--270.
       \bibitem {Lecko1988} A. Lecko and D. Partyka, An alternative proof of a result due to Douady and Earle, Ann. Univ. Mariae Curie-Sklodowska Sec. A {\bf 42} (1988), 59--68.
       \bibitem {Mane1983} R. Man\'e, P. Sad, and D. Sullivan, On the dynamics of rational maps, Ann. Sci. {101}c. Norm. Sup. {\bf 16} (1983), 193--217.
       \bibitem {Mitra2010} S. Mitra and H. Shiga, Extensions of holomorphic motions and holomorphic families of M\"obius groups, Osaka J. Math. {\bf 47} (2010), 1167--1187.
       \bibitem {Nag1988} S. Nag, The Copmlex Analytic Theory of Teichm\"uller Spaces, JOHN WILEY and SONS, 1988.
       \bibitem {Pommerenke1991} Ch. Pommerenke, Boundary Behaviour of Conformal Maps, Springer-Verlag, 1991.
       \bibitem {Sario1970} L. Sario and M. Nakai, Classification Theory of Riemann Surfaces, Springer-Verlag, Berlin-Heidelberg-New York, 1970.
       \bibitem {Shiga2013} H. Shiga, On injectivity radius in configuration space and in moduli space, Contemp. Math. {\bf 590} 2013, 183--189.
       \bibitem {Shiga2016} H. Shiga, On analytic properties of deformation spaces of Kleinian groups, Trans. Amer. Math. Soc {\bf 368} (2016), no. 9, 6627--6642.
       \bibitem {Slodkowski1991} Z. Slodkowski, Holomorphic motions and polynomial hulls, Proc. Am. Math. Soc. {\bf 111}
(1991), no. 2, 347--355.
\bibitem {Slodkowski1995} Z. Slodkowski, Extensions of holomorphic motions, Ann. della Sc. Norm. Super. di Pisa, Cl.di Sci. 22 (1995), no. 2, 185--210.
\bibitem {Sullivan1986} D Sullivan and W. P. Thurston, Extending holomorphic motions, Acta Math. 157 (1986), no. 1, 243--257.
\end{thebibliography}
\end{document}